\documentclass[a4paper,12pt,reqno]{amsart}

\usepackage{amsmath, wrapfig}
\usepackage{dsfont, a4wide, amsthm, amssymb, amsfonts, graphicx,amscd}
\usepackage{fancyhdr, xspace, psfrag, setspace, supertabular, color}
\usepackage{hyperref}
\hypersetup{
     colorlinks=true}
\usepackage{bbm}
\usepackage{stmaryrd}
\usepackage{txfonts}
\usepackage{enumitem}
\usepackage[foot]{amsaddr}

\newtheorem{theorem}{Theorem}[section]

\newtheorem{lemma}[theorem]{Lemma}

\newtheorem{corollary}[theorem]{Corollary}
\newtheorem{conjecture}[theorem]{Conjecture}
\newtheorem{definition}[theorem]{Definition}
\newtheorem*{definition*}{Definition}
\newtheorem{problem}[theorem]{Problem}

\newtheorem{question}[theorem]{Question}

\def\e{\epsilon}
\def\E{\mathbb{E}}

\def\R{\mathbb{R}}

\def\N{\mathbb{N}}
\def\P{\mathbb{P}}
\def\F{\mathbb{F}}

\def\a{\alpha}
\def\be{\beta}

\def\g{\gamma}
\def\d{\delta}

\def\b1{\mathbbm{1}}
\def\AA{\b1_A}
\def\BB{\b1_B}

\def \hcA {{\widehat {\b1_{\mathcal A}}}}

\def \hC {{\hat C}}

\def \cA{\mathcal A}
\def \cB{\mathcal B}
\def \cC{\mathcal C}
\def \cD{\mathcal D}
\def \cG{\mathcal G}
\def \cW{\mathcal W}

\def \codim{{\rm codim}}

\def \mn{\mathrm{M}_{m,n}(\F_2)}

\def \s{\subset}

\begin{document}

\title{Subsets of Cayley graphs that induce many edges}

\author{W.T. Gowers}
\address{Department of Pure Mathematics and Mathematical Statistics, Wilberforce Road, Cambridge CB3~0WB, UK.} 
\email[W. T. Gowers]{w.t.gowers@dpmms.cam.ac.uk}
\thanks{The first author is supported by a Royal Society 2010 Anniversary Research Professorship. He also held a Research Chair of the Fondation Sciences Math\'ematiques de Paris while this research was carried out.}


\author{O. Janzer}
\email[O. Janzer]{oj224@cam.ac.uk}

\maketitle

\begin{abstract}
Let $G$ be a regular graph of degree $d$ and let $A\subset V(G)$. Say that $A$ is \emph{$\eta$-closed} if the average degree of the subgraph induced by $A$ is at least $\eta d$. This says that if we choose a random vertex $x\in A$ and a random neighbour $y$ of $x$, then the probability that $y\in A$ is at least $\eta$. The work of this paper was motivated by an attempt to obtain a qualitative description of closed subsets of the Cayley graph $\Gamma$ whose vertex set is $\F_2^{n_1}\otimes \dots \otimes \F_2^{n_d}$ with two vertices joined by an edge if their difference is of the form $u_1\otimes \cdots \otimes u_d$. For the matrix case (that is, when $d=2$), such a description was obtained by Khot, Minzer and Safra, a breakthrough that completed the proof of the 2-to-2 conjecture. In this paper, we formulate a conjecture for higher dimensions, and prove it in an important special case. Also, we identify a statement about $\eta$-closed sets in Cayley graphs on arbitrary finite Abelian groups that implies the conjecture and can be considered as a ``highly asymmetric Balog-Szemer\'edi-Gowers theorem" when it holds. We conclude the paper by showing that this statement is not true for an arbitrary Cayley graph. It remains to decide whether the statement can be proved for the Cayley graph $\Gamma$.
\end{abstract}

\section{Introduction}

The Unique Games Conjecture, formulated by Khot \cite{khot} in 2002, is a central conjecture in theoretical computer science. If true, it implies that for a wide class of natural problems it is NP-hard to find even a very crude approximate solution in polynomial time. Recently, a weakening of the conjecture known as the 2-to-2 Games Conjecture, where the approximation is required to be less crude (so it is easier to prove hardness) was proved by Khot, Minzer and Safra \cite{KMS}, a result that is considered as a major step towards the Unique Games Conjecture itself. More precisely, after work by various authors, the problem had been reduced to a statement about a certain Cayley graph, and Khot, Minzer and Safra proved that statement.

The Cayley graph $\Gamma$ in question has as its vertex set the set of all $m\times n$ matrices over $\F_2$, with two vertices joined by an edge if their difference has rank 1. Let us say that a subset $\mathcal A\subset\mn$ is \emph{$\eta$-closed} if the probability that $A+B\in\cA$, when $A$ is chosen uniformly from $\cA$ and $B$ is chosen uniformly from all rank-1 matrices, is at least $\eta$. In graph terms, this is the probability that a random neighbour of a random point in $\cA$ is itself in $\cA$. 

A simple example of an $\eta$-closed set is the set $\{A\in\mn:Ax=y\}$, for some pair of vectors $x\in \F_2^n,y\in\F_2^m$. Indeed, if $Ax=y$ and $B$ is a random matrix of rank 1, then $x\in\ker B$ with probability roughly 1/2. 
But if $x\in\ker B$, then $(A+B)x=y$, so $A+B\in\cA$ as well. A very similar, but distinct, example is the set $\{A\in\mn:A^Tx=y\}$. Let us call sets of one of these two kinds \emph{basic sets}.

We can form further examples by taking intersections of a small number of basic sets. For example, if $x_1,\dots,x_k$ are linearly independent and we take a set of the form 
\[\{A\in\mn:Ax_1=y_1,\dots,Ax_k=y_k\},\] 
then with probability approximately $2^{-k}$ each $x_i$ belongs to $\ker B$, so for any $A$ in the set, $A+B$ belongs to the set with probability approximately $2^{-k}$. The result of Khot, Minzer and Safra is the following.

\begin{theorem} \label{KMS}
For every $\eta>0$ there exist $\delta>0$ and a positive integer $k$ such that if $\cA$ is any  $\eta$-closed subset of $\mn$, then there exists an intersection $\cC$ of at most $k$ basic sets such that $|\cA\cap \cC|\geq\d|\cC|$ and $\cC\neq \emptyset$.
\end{theorem}

\noindent In other words, every closed set is dense inside some intersection of a small number of basic sets.

It is well known and not hard to see that this in fact leads to a characterization (at least qualitatively) of  closed sets. Indeed, observe first that if $\cA$ is  $\eta$-closed, then the subgraph induced by $\cA$ has average degree at least $\eta|\cB|$, where $\cB$ is the set of rank-1 matrices, and maximal degree at most $|\cB|$.
Therefore, any subset of $\cA$ of size at least $(1-\eta/4)|\cA|$ has average degree at least $\eta|\cB|/2$. It follows from this observation and Theorem \ref{KMS} that we can find disjoint subsets $\cA_1,\dots,\cA_r$ of~$\cA$, subsets $\cC_1,\dots,\cC_r$ of $\mn$, a positive real number $\d=\d(\eta)$ and a positive integer $k=k(\eta)$ with the following properties.
\begin{enumerate}
\item The sets $\cA_i$ are disjoint.
\item Each $\cC_i$ is an intersection of at most $k$ basic sets.
\item For each $i$, $|\cA_i\cap\cC_i|\geq\d|\cC_i|$.
\item $|\bigcup_i\cA_i|\geq\eta|\cA|/4$.
\end{enumerate}

Conversely, if such sets exist, then the probability that a random matrix $A\in\cA$ belongs to some $\cA_i$ is at least $\eta/4$. If it belongs to $\cA_i$, then we can use the following lemma. We write $u\otimes v$ for the rank-1 matrix $M$ with $M_{ij}=u_iv_j$, which sends a vector $x$ to the vector $\langle x,v\rangle u$. Note also that $(u\otimes v)^T$ sends $x$ to $\langle x,u\rangle v$.

\begin{lemma}
Let $\cC$ be an intersection of at most $k$ basic sets and let $\cA\subset\cC$ be a subset of relative density at least $\d$. Then $\cA$ is $2^{-k}(\d-2^{-(m-k)})$-closed.
\end{lemma}

\begin{proof}
Let us set $\cC(x,y)=\{A\in\mn:Ax=y\}$, and $\cC'(x,y)=\{A\in\mn:A^Tx=y\}$. Let $x_1,\dots,x_k,y_1,\dots,y_k$ be non-zero vectors such that $\cC=\bigcap_{i=1}^r\cC(x_i,y_i)\cap\bigcap_{i=r+1}^k\cC'(x_i,y_i)$.

Let $u\otimes v$ be a rank-1 matrix. If there exists $i\leq r$ such that $\langle x_i,v\rangle\ne 0$, then $(A+u\otimes v)(x_i)=y_i+u$, so $A+u\otimes v\notin\cC_i$ and hence $A+u\otimes v\notin\cC$. Similarly, if there exists $i>r$ such that $\langle x_i,u\rangle\ne 0$, then $(A+u\otimes v)^T(x_i)=y_i+v$ and again $A+u\otimes v\notin\cC$. 

We shall now bound from below the probability that $A+u\otimes v\in\cA$ given that $A\in\cA$ and that $\langle x_i,v\rangle=0$ for every $i\leq r$ and $\langle x_i,u\rangle=0$ for every $r<i\leq k$, noting that the condition on $u\otimes v$ states that $(u,v)\in U\times V$ for a pair of subspaces $U$ and $V$ with codimensions that add up to at most $k$, a condition that occurs with probability $2^{-k}$.

Let us now condition further on the choice of $v\in V$. That means that we fix $v$, choose a random $u\in U$, and add $u\otimes v$ to $A$. If we allow $u$ to take the value 0, then the resulting matrix is uniformly distributed in the affine subspace $A+U\otimes v$, so the probability that it is in $\cA$ is equal to the relative density of $\cA$ inside this affine subspace.

The translates of $U\otimes v$ by matrices in $\cC$ partition $\cC$. Let us write them as $\cW_1,\dots,\cW_s$, and let the relative density of $\cA$ inside $\cW_i$ be $\d_i$. Then, still fixing $v$, we have that
\[\P[A+u\otimes v\in\cA]=\sum_i\,\P[A\in\cW_i]\,\P[A+u\otimes v\in\cA|A\in\cW_i]=\sum_i \frac{\d_i^2}{s\delta}\geq\d.\]
This statement is true regardless of $v$, so we deduce that the probability that $A+u\otimes v\in\cA$ given that $A\in\cA$ and $(u,v)\in U\times(V\setminus\{0\})$ is at least $\d$. If we now insist that $u\ne 0$, we reduce this probability by at most $2^{-(m-k)}$, so the result is proved.
\end{proof}

Let $B\in\cB$ be chosen uniformly at random. Given the lemma above, applied to the sets $\cA_i$ and $\cC_i$, we deduce that the conditional probability that $A+B\in\cA_i$ given that $A\in\cA_i$ is at least $c(\d,k)$, and from that it follows that $\cA$ is  $c(\d,k)\eta/4$-closed.

Thus, a set $\cA$ is $\eta$-closed for some not too small $\eta$ if and only if an appreciable fraction of $\cA$ is efficiently covered by disjoint intersections of few basic sets.

\medskip

Barak, Kothari and Steurer suggest in \cite{barak} that establishing a higher dimensional analogue of Theorem \ref{KMS} may be a useful step in obtaining a proof of the full Unique Games Conjecture, though they do not actually provide a formal reduction. The main purpose of this paper is to formulate a suitable conjecture and prove some partial results towards it. We say that $\cA\s \F_2^{n_1}\otimes \dots \otimes \F_2^{n_d}$ is $\eta$-closed if with probability at least $\eta$, we have $A+u_1\otimes \dots \otimes u_d\in \cA$, when $A\in \cA$ and vectors $u_i\in \F_2^{n_i}\setminus \{0\}$ are chosen independently and uniformly at random. 

\begin{problem} \label{tensorproblem}
	Give a qualitative description of $\eta$-closed sets $\cA\s \F_2^{n_1}\otimes \dots \otimes \F_2^{n_d}$.
\end{problem}

To see that this is indeed a generalization of the problem about matrices considered above, we identify $\mn$ with $\F_2^m \otimes \F_2^n$ in the usual way, which leads to a slight reformulation of Theorem \ref{KMS} in terms of tensor products. Note first that under this identification, the set $\{M\in \mn:Mx_1=\dots=Mx_a=M^Ty_1=\dots=M^Ty_b=0\}$ becomes the set $H\otimes K$, where $H=\langle y_1,\dots,y_b\rangle^{\perp}$ and $K=\langle x_1,\dots,x_a\rangle^{\perp}$. It follows that an intersection of at most $k$ basic sets is either empty or a translate of $H\otimes K$ for some pair of subspaces $H\s \F_2^m,K\s \F_2^n$ with $\codim(H)+\codim(K)\leq k$.

In the higher-dimensional case, there is a richer class of sets $\cA\s \F_2^{n_1}\otimes \dots \otimes \F_2^{n_d}$ that are $\eta$-closed. To describe them, we introduce the following piece of notation, which we shall use repeatedly in the rest of the paper. Given a non-empty subset $I\s \lbrack d\rbrack$, write $\F_2^{I}$ for $\bigotimes_{i\in I} \F_2^{n_i}$, so that we naturally have $\F_2^{n_1}\otimes \dots \otimes \F_2^{n_d}=\F_2^{I}\otimes \F_2^{I^c}$. Here and elsewhere in the paper, $I^c$ stands for $\lbrack d\rbrack \setminus I$.

Now we say that $\cC$ is $k$-\emph{simple} if there exists a collection of subspaces $H_I\s\F_2^I$ of codimension at most $k$, one for each non-empty subset  $I\s \lbrack d\rbrack$, such that $\cC$ is a translate of the set $\bigcap_{I\s \lbrack d\rbrack,I\neq \emptyset}(H_I\otimes \F_2^{I^c})$ (where $H_{\lbrack d \rbrack}\otimes \F_2^{\emptyset}$ is just $H_{\lbrack d\rbrack}$), which is a subspace of $\F_2^{n_1}\otimes \dots \otimes \F_2^{n_d}$. It is not hard to see that this subspace contains at least a proportion $c(d,k)>0$ of all rank-1 tensors $u_1\otimes \dots \otimes u_d$ (provided that $n_1,\dots,n_d$ are sufficiently large), so it is $c(d,k)$-closed. It follows that any translate of it is $c(d,k)$-closed too. 

We now make the following conjecture.

\begin{conjecture} \label{mainconjecture}
	For any $\eta>0$ and any positive integer $d$, there exist $k=k(d,\eta)$ and $\rho=\rho(d,\eta)>0$ with the following property. Let $\cA\s \F_2^{n_1}\otimes \dots \otimes \F_2^{n_d}$ be $\eta$-closed. Then there is a $k$-simple set $\cC$ such that $|\cA\cap \cC|\geq \rho|\cC|$.
\end{conjecture}

\noindent Note that in the $d=2$ case, we allow translates of sets $(H_{\{1\}}\otimes H_{\{2\}})\cap H_{\{1,2\}}$ rather than just translates of $H_{\{1\}}\otimes H_{\{2\}}$, so Conjecture \ref{mainconjecture} might seem to be weaker than Theorem \ref{KMS}. However, this actually makes no difference, since when intersecting with $H_{\{1,2\}}$, the cardinality of the set drops by a factor at most $2^k$.

The main result of this paper, stated later in this section, is a proof of Conjecture \ref{mainconjecture} in an important special case.


\subsection{What can be said about more general Cayley graphs?}

It is tempting to try to prove Conjecture \ref{mainconjecture} by identifying and proving a statement that applies to a much wider class of Cayley graphs, of which Conjecture \ref{mainconjecture} would be a special case. We would begin with an Abelian (or even non-Abelian) group $G$ and a pair of subsets $A,B\subset G$, where we think of $B$ as the set of generators, satisfying the hypothesis that $|\{(a,b)\in A\times B: a+b\in A\}|\geq\eta|A||B|$. We shall say in this situation that $A$ is $(B,\eta)$-\emph{closed} (in $G$). 

Another way of writing the condition is
\[\langle \AA*\mu_B,\AA\rangle\geq\eta\a,\]
where $\a$ is the density of $A$, $\mu_B$ is the characteristic measure of $B$ (that is, the function that takes the value $|G|/|B|$ on $B$ and 0 elsewhere) and we define $f*g(x)$ to be $\E_{y+z=x}f(y)g(z)$. By the Cauchy-Schwarz inequality the left-hand side is at most $\|\AA*\mu_B\|_2\|\AA\|_2=\|\AA*\mu_B\|_2\a^{1/2}$, where inner products and $L_p$ norms are defined using expectations, so our hypothesis implies that $\|\AA*\mu_B\|_2^2\geq\eta^2\a$. It is easy to see that this ``mixed energy" $\|\AA*\mu_B\|_2^2$ can be at most $\a$, with equality if and only if $a+b-b'\in A$ for every $a\in A, b,b'\in B$. 

At this point let us recall the so-called asymmetric Balog-Szemer\'edi-Gowers theorem, which can be found in \cite{taovu} as Theorem 2.35. (For a useful alternative presentation of the theorem, see also \cite{green}.) The main assumption of the theorem is that $A,B$ are two finite subsets of an Abelian group, with densities $\a$ and $\be$, such that $\|\AA*\BB\|_2^2\geq\eta\a\be^2$ (which is equivalent to saying that $\|\AA*\mu_B\|_2^2\geq\eta\a$), but there is also an assumption that $A$ is not too much bigger than $B$. The precise statement is as follows.

\begin{theorem} \label{asym}
For every $\e>0$ there exists a constant $C=C(\e)$ with the following property. Let $G$ be a finite Abelian group, let $L\geq 1$, let $0<\eta\leq 1$ and let $A$ and $B$ be finite subsets of $G$ with densities $\a$ and $\be$, such that $\a\leq L\be$ and $\|\AA*\BB\|_2^2\geq 2\eta\a\be^2$. Then there exist a subset $H\subset G$ such that $|H+H|\leq C\eta^{-C}L^\e|H|$, a subset $X\subset G$ of size at most $C\eta^{-C}L^\e|A|/|H|$ such that $|A\cap(X+H)|\geq C^{-1}\eta^CL^{-\e}|A|$, and some $x\in G$ such that $|B\cap(x+H)|\geq C^{-1}\eta^CL^{-\e}|B|$.
\end{theorem}

\noindent More qualitatively speaking, if $A$ is not too much larger than $B$ and $\|\AA*\BB\|_2^2$ is within a constant of its largest possible value, then there is a set $H$ of small doubling such that a small number of translates of $H$ cover a substantial proportion of $A$, and some translate of $H$ covers a substantial proportion of $B$. It is not hard to see that the converse holds as well.

This theorem cannot be used to prove Conjecture \ref{mainconjecture} because of the condition that $\a\leq L\be$, which does not apply here since the set $\cA$ in Conjecture \ref{mainconjecture} can be much bigger than the set $\cB$. That raises the following question, which generalizes Problem \ref{tensorproblem}.

\begin{question} \label{generalq}
Let $G$ be a finite Abelian group, let $\eta>0$, and let $A,B\subset G$ be subsets such that $A$ is $(B,\eta)$-closed in $G$. What can be said about $A$, $B$ and the relationship between them?
\end{question}

\noindent A similar question can of course be asked with the slightly weaker hypothesis that $\|\AA*\BB\|_2^2\geq\eta^2\a \be^2$, but we shall concentrate on the question as stated, since it is more closely related to Conjecture \ref{mainconjecture}.

An immediate observation is that we cannot hope to say anything interesting about the structure of $B$, even if $\eta=1$. For example, $\eta=1$ if $A=G$ and $B$ is an arbitrary subset of $G$. For a more general example, one can let $A$ be an arbitrary union of cosets of some subgroup $H$ and let $B$ be an arbitrary subset of $H$. For a slightly different example, let $G=\F_2^n$, let $B$ be the set $\{e_1,\dots,e_n\}$ of standard basis vectors, and let $A$ be a union of $n/3$-dimensional affine subspaces $V_i$, such that each $V_i$ is a random translate of the subspace generated by $n/3$ randomly chosen $e_j$. Then if $x\in V_i$ and $b\in B$, the probability that $x+b\in V_i$ is 1/3, so $A$ is 1/3-closed.



Any general statement will have to be weak enough to allow for examples like these. The last example shows that we cannot hope to find a single set $H$ of small doubling and cover a large portion of $A$ efficiently with translates of $H$, unless $H$ is of constant size, in which case the conclusion becomes trivial. To sketch briefly why not, observe first that by Freiman's theorem we can assume that $H$ is a subspace. Next, note that for each vector $x$, the probability that it belongs to the span of a random $n/3$ standard basis vectors is exponentially small in the size of the support of $x$. We can also use the following simple lemma.

\begin{lemma}
Let $H$ be a subspace of dimension $d$ and let $k\leq d$. Then the number of vectors in $H$ of support size at most $k$ is at most the number of vectors of support size at most $k$ in the subspace generated by the first $d$ standard basis vectors $e_1,\dots,e_d$, namely $\sum_{i=0}^k\binom di$.
\end{lemma}

\begin{proof}
Let $u_1,\dots,u_d$ be a basis for $d$. By Gaussian elimination, we can convert $u_1,\dots,u_d$ into a basis $v_1,\dots,v_d$ and find coordinates $t_1,\dots,t_d$ such that $v_i(t_j)=\d_{ij}$. Then the support size of $\sum_i\lambda_iv_i$ is at least the number of non-zero $\lambda_i$, which proves the result.
\end{proof}

When $d$ is large, it follows that the proportion of vectors in $H$ of small support is very small. Combining these observations, one can show that for every $\eta$ there exists $d$ such that if $H$ is a $d$-dimensional subspace, then the probability that a random subspace $V$ of dimension $n/3$ is $(H,\eta/2)$-closed is at most $\eta/2$. This in turn can be used to prove that with high probability the set $A$ described above (for a suitable number of $V_i$) is not $(H,\eta)$-closed for any $H$ of dimension $d$ or above.

However, these examples do not rule out a weakening along the following lines.



\begin{question}\label{wrongq}
Let $G$ be a finite Abelian group, let $\eta>0$, and let $A,B$ be subsets of $G$ such that $A$ is $(B,\eta)$-closed in $G$. Does it follow that $A$ has a non-empty subset $A'$ such that $A'$ is $(B,\eta')$-closed in $G$, and $|A'+A'|\leq C|A'|$, where $\eta'>0$ and $C$ are constants that depend on $\eta$ only?
\end{question}

\noindent An argument similar to the one we mentioned just after the statement of Theorem \ref{KMS} shows that if the answer is yes, then we can find a collection of disjoint subsets $A_1,\dots,A_m$ that cover a substantial proportion of $A$, each one with small doubling and each one $(B,\eta')$-closed (with a slightly smaller $\eta'$). Thus, we would be able to obtain a conclusion similar to that of Theorem \ref{asym} but without the requirement that the structured sets are all translates of one another.

A positive answer would also imply Conjecture \ref{mainconjecture}. Indeed, by Freiman's theorem $A_i$ is contained in a subspace $V_i$ not much larger than $A_i$. This reduces the conjecture to the case where $\cA$ is a subspace. In that case, a very simple corollary of our main result, Corollary \ref{spacecorollary} (stated later) proves the conjecture.


However, the answer to Question \ref{wrongq} is easily seen to be negative (which implies that it is also negative if we assume the weaker mixed-energy hypothesis instead). The example we are about to give was communicated to us privately by Boaz Barak as a counterexample to a related but slightly different statement. 

For convenience let $n$ be odd, let $A\subset\F_2^n$ be the set of all vectors with $(n\pm 1)/2$ coordinates equal to 1, and let $B$ be the set of standard basis vectors. Then it is easy to see that $A$ is $\eta$-closed for $\eta=(n+1)/2n\approx 1/2$. Suppose now that we could find a subset $A'\subset A$ such that $|A'+A'|\leq C|A'|$, and $A'$ is $(B,\eta')$-closed. By Freiman's theorem, $A'$ is contained in a subspace $V$ that is not much bigger than $A'$, which implies that $V$ is $c$-closed for some positive constant $c=c(\eta)$. That implies that at least $cn$ of the standard basis vectors belong to $V$. Let $W$ be the subspace spanned by these basis vectors. The maximum number of elements of $A$ that can belong to a translate $x+W$ of $W$ is $2(cn)^{-1/2}|W|$, and therefore $|A'|\leq 2(cn)^{-1/2}|V|$. This contradicts the fact that $V$ is not much bigger than $A'$.

In this paper we formulate a yet weaker conjecture and prove that it still implies Conjecture \ref{mainconjecture}. Unfortunately, we also give a counterexample to the weaker conjecture. The counterexample does not make the implication 
vacuous, however, because the implication depends on a non-trivial theorem that is true and of some interest: it is just that for a general Cayley graph (on a finite Abelian group) one cannot deduce the hypotheses of the theorem from the assumption that a set is $\eta$-closed. It is conceivable that one might be able to prove Conjecture \ref{mainconjecture} (and thereby also give a different proof of the theorem of Khot, Minzer and Safra) by using additional properties of the particular Cayley graph that that conjecture is about.

How, then, might one try to find a conjecture that would not be contradicted by the ``two-layers" example just discussed? One observation that suggests a possible way forward is the following. Suppose that we extend the set by adding a few more layers. If, say, we take not just the middle two layers but the middle $\e^{-1}$ layers (or thereabouts), then we obtain a new set inside which the first set has relative density approximately $2\e$, and this new set is  $(1-2\e)$-closed, since a random element of the set will be in one of the interior layers with probability approximately (and in fact slightly bigger than) $1-2\e$, and adding an arbitrary basis vector to such an element will give another element of the set.

So perhaps we could hope that if $A$ is $(B,\eta)$-closed, then there is a set $C$ that is  $(B,1-\e)$-closed such that $|A\cap C|\geq\d|C|$ for some $\d$ that depends on $\eta$ and $\e$ only.

However, simple modifications of the example show that this is too much to ask. For instance, we can take as our set $A$ the set of all $x\in\F_2^n$ such that $m$ or $m+1$ coordinates are equal to 1 and all but the first $2m$ coordinates are zero. If $m$ is around $n/4$, say, then the resulting set is  $(B,1/4)$-closed, but there is no prospect of $A$ living densely in a set that is almost perfectly closed, because of the need to add basis vectors corresponding to coordinates beyond $2m$. 

A further example to consider is the set of all $x\in\F_2^n$ such that at most $n/3$ coordinates are equal to 1. This set is  $(B,1/3)$-closed (at least -- in fact it is more like $(B,2/3)$-closed because the probability that a random element of the set has exactly $\lfloor n/3\rfloor$ coordinates equal to 1 is approximately 1/2), but for similar reasons to the previous example, one cannot find an almost perfectly closed set with a significant proportion of its elements in the set.

However, the picture changes if we ask for slightly less. Let us informally call a set $C$ \emph{good} if there is a proportional-sized subset $B'\subset B$ such that $C$ is $(B',1-\e)$-good for some small constant $\e$. Thus, now we ask only that $C$ should be almost closed for a \emph{large subset} of $B$ rather than for the whole set. 

It is not immediately clear how to use this definition, because the statement that $|A\cap C|\geq\d|C|$ for a good set $C$ can be true for uninteresting reasons. For example, we could take $C$ to be the union of a subspace $V$ generated by $n/5$ basis vectors together with an arbitrary subset of $A$ of cardinality $2\d|V|$. To remedy this, we insist that $C$ is ``related to $A$" in the graph in a different sense from that of $A$ being dense in $C$.





Here, then, is a question that replaces Question \ref{wrongq}. 


\begin{question} \label{mainq}
Is it true that for every $\eta,\e>0$ there exist $c>0,\d>0$ and positive integer $l$ with the following property? Let $G$ be a finite Abelian group and let $A,B\subset G$ be subsets such that $A$ is $(B,\eta)$-closed. Then there is a subset $B'\subset B$ and a non-empty subset $C\subset G$ with the following properties.
\begin{enumerate}
\item $|B'|\geq\d|B|$.
\item $C$ is  $(B',1-\e)$-closed.
\item $C\subset \bigl\{x\in G: \AA*\mu_B*\dots*\mu_B*\mu_{-B}*\dots*\mu_{-B}(x)\geq c\bigr\}$ where the number of $\mu_B$s and $\mu_{-B}$s in the convolution is $l$.
\end{enumerate}
\end{question}

\noindent Condition (3) is saying that for any $x\in C$, the probability that $x-b_1-\dots-b_l+b_{l+1}+\dots+b_{2l}\in A$, when the $b_i$ are chosen uniformly and independently at random from $B$, is at least $c$. When the group $G$ is $\F_2^n$ for some $n$, we can and will simplify it, since $B=-B$.

To see that this question improves on Question \ref{wrongq}, let us consider the two problematic examples for that question. If $m$ is odd and $A\subset\F_2^n$ consists of all sequences with $(m\pm 1)/2$ 1s and with no 1s after the $m$th coordinate, then let $C$ be the set of all sequences with no 1s after the $m$th coordinate that have between $(m-1)/2-\e^{-1}$ and $(m+1)/2+\e^{-1}$ 1s. If $l=\e^{-1}$, then for any $x\in C$, the probability that $x-b_1-\dots-b_l\in A$ is at least $(\frac{m}{2n})^l$. (This is because conditional on $b_i\in \{e_1,\dots,e_m\}$ this probability is at least $\frac{1}{2^l}$.) Moreover, if $B'=\{e_1,\dots,e_m\}$, then for every $b\in B'$ and every $c\in C$ that is not on the boundary (in the obvious sense), we have that $b+c\in C$, so $C$ is  $(B',1-\e)$-closed.

Now let us look at the example where $A$ is the set of all sequences with at most $n/3$ 1s. This time let $C$ be the set of all sequences that are 0 after the first $2n/3$ coordinates and have at most $n/3+\e^{-1}$ 1s, and let $B'=\{e_1,\dots,e_{2n/3}\}$. Then for any $x\in C$, the probability that $x-b_1-\dots-b_l\in A$ is at least $(\frac{1}{3})^l$, where $l=\e^{-1}$. Moreover, $C$ is  $(1-\e)$-closed, again because adding an element of $B'$ to a non-boundary element of $C$ gives an element of $C$.



\subsection{Our main result}

Let us now see why a positive answer to Question \ref{mainq} would imply Conjecture \ref{mainconjecture}. The deduction will be easy once we have established the following theorem, which is the main result of this paper. In the statement of the theorem, and in the rest of this paper, $\cG$ denotes $\F_2^{n_1}\otimes \dots \otimes \F_2^{n_d}$ and $\cB$ denotes the multiset $\{u_1\otimes \dots \otimes u_d: u_i\in \F_2^{n_i} \text{ for all } i\}$ (which is a multiset only because some of the $u_i$ can be zero). Note that the notion of $(B,\eta)$-closedness can be generalized in an obvious way to multisets.

\begin{theorem} \label{mainresult}
	For any $\theta>0$, there exists $\e=\e(d,\theta)>0$ with the following property. Let $\delta>0$. Then there exists a positive integer $k=k(d,\d)$ with the following property. For any $\cB'\subset \cB$ with $|\cB'|\geq \d |\cB|$ and any $\cA\subset \cG$ which is  $(\cB',1-\e)$-closed, there exists a $k$-simple set $\cD\subset \cG$ such that $|\cD\cap \cA|\geq (1-\theta)|\cD|$.
\end{theorem}

\noindent It is convenient to state the following corollary separately, which follows from Theorem \ref{mainresult} by taking $\theta=1/2$.

\begin{corollary} \label{maincorollary}
	There exists $\e=\e(d)>0$ such that for any $\delta>0$, there exists a positive integer $k=k(d,\d)$ with the following property. For any $\cB'\subset \cB$ with $|\cB'|\geq \d |\cB|$ and any $\cA\subset \cG$ which is  $(\cB',1-\e)$-closed, there exists a $k$-simple set $\cD\subset \cG$ which has $|\cD\cap \cA|\geq \frac{1}{2}|\cD|$.
\end{corollary}

\noindent Let us see why Conjecture \ref{mainconjecture} follows from Corollary \ref{maincorollary} and a positive answer to Question \ref{mainq} in the case of the group $\cG$ and the subset $B\s\cG$ of rank-1 tensors. Let $\eta>0$. Pick $\e=\e(d)$ so that the conclusion of Corollary \ref{maincorollary} holds. If the answer to Question \ref{mainq} is positive for $\cG$ and $B$, then we can choose $c>0,\d>0$, and a positive integer $l$ such that the conclusion of the question is true. Now let $\cA\s \cG$ be $\eta$-closed. This is saying that $\cA$ is $(B,\eta)$-closed. By the conclusion of Question \ref{mainq}, there exist a set $B'\s B$ with $|B'|\geq \d|B|$ , and a non-empty subset $\cC\s \cG$ such that $\cC$ is $(B',1-\e)$-closed and $\cC\subset \bigl\{x\in \cG: \b1_{\cA}*\mu_B*\dots*\mu_B(x)\geq c\bigr\}$, where the number of $\mu_B$s in the convolution is $l$. Define $\cB'$ to be the multiset that consists of the set $B'$ together with the multiset of all $u_1\otimes \dots \otimes u_d$ with $u_i\in \F_2^{n_i}$ for each $i$ and with at least one $u_i$ equal to $0$. Note that $|\cB'|\geq \d|\cB|$ and $\cC$ is $(\cB',1-\e)$-closed. By Corollary \ref{maincorollary}, there exists a $k$-simple set $\cD\s \cG$, for some $k=k(d,\d)$, which has $|\cD\cap \cC|\geq \frac{1}{2}|\cD|$. Now pick $x\in \cD$ and $b_1,\dots,b_l \in B$ uniformly and independently at random. The probability that $x-b_1-\dots-b_l\in \cA$ is at least $c/2$. Therefore, there exists some $y\in \cG$ such that when $x\in \cD$ is randomly chosen, the probability that $x-y\in \cA$ is at least $c/2$. That is, $|(\cD-y)\cap \cA|\geq \frac{1}{2}c|\cD|=\frac{1}{2}c|\cD-y|$. But $\cD-y$ is a $k$-simple set, which finishes the proof of Conjecture \ref{mainconjecture}.

Another simple corollary of Theorem \ref{mainresult} is the following result, which is Conjecture \ref{mainconjecture} in the case where $\cA$ is a subspace.

\begin{corollary} \label{spacecorollary}
	For any $\eta>0$ and any positive integer $d$, there exists some $k=k(d,\eta)$ with the following property. If $V\s \cG$ is a subspace which is $\eta$-closed, then $V$ contains a $k$-simple subspace. That is, $V\supset \bigcap_{I\s \lbrack d\rbrack,I\neq \emptyset}(H_I\otimes \F_2^{I^c})$ for some collection of subspaces $H_I\s \F_2^{I}$ of codimension at most $k$.
\end{corollary}

\begin{proof}
	Since $V$ is a vector space, the condition that $V$ is $\eta$-closed says that $u_1\otimes \dots \otimes u_d\in V$ for at least a proportion of $\eta$ of all rank-1 tensors $u_1\otimes \dots \otimes u_d$. Thus, there exists some $\cB'\s \cB$ with $|\cB'|\geq \eta |\cB|$ such that $V$ is $(\cB',1)$-closed. Taking $\theta$ sufficiently close to $0$ in Theorem \ref{mainresult}, it follows that $V\supset \cD$ for a $k$-simple set $\cD$, where $k$ depends only on $d$ and $\eta$. Then $\cD$ is a translate of $\bigcap_{I\s \lbrack d\rbrack,I\neq \emptyset}(H_I\otimes \F_2^{I^c})$ for some $H_I\s \F_2^{I}$ of codimension at most $k$. Since $V$ is a vector space, it follows that $V\supset \bigcap_{I\s \lbrack d\rbrack,I\neq \emptyset}(H_I\otimes \F_2^{I^c})$.
\end{proof}

\noindent In the next section, we shall prove Theorem \ref{mainresult}. In the last section, we show that the answer to Question \ref{mainq} is negative.


\section{The proof of Theorem \ref{mainresult}}


Note that $\cG=\F_2^{n_1}\otimes\dots\otimes\F_2^{n_d}$ can be viewed as the set of $d$-dimensional $(n_1,\dots,n_d)$-arrays over $\F_2$ which in turn can be viewed as $\F_2^{n_1n_2\dots n_d}$, equipped with the entry-wise dot product.


The proof of Theorem \ref{mainresult} will be reasonably simple once we have established the following result. In the statement of this lemma, and in the rest of this section, we write $kB$ to mean the set of elements of $\cG$ that can be written as a sum of \textit{at most} $k$ elements of $B$, where $B$ is some fixed (multi)subset of $\cG$.


\begin{lemma} \label{mainlemma}
	For all $\d>0$ and $d\in\N$ there exist $f_1(d),f_3(d)>0$ and $f_2(d,\d)\in\N$ with the following property. Let $\cB'\subset \cB$ be a multiset such that $|\cB'|\geq \d |\cB|$. Then there exists a multiset $Q$ whose elements are chosen from $f_1(d)\cB'$ (but with arbitrary multiplicity) with the following property. The set of arrays $r\in \cG$ with $r.q=0$ for at least $(1-f_3(d))|Q|$ choices $q\in Q$ is contained in $\sum_{I\subset \lbrack d\rbrack, I\neq \emptyset} V_I\otimes \F_2^{I^c}$ for a collection of subspaces $V_I\subset \F_2^{I}$ that each have dimension at most $f_2(d,\d)$.
\end{lemma}

In order to deduce Theorem \ref{mainresult} from this lemma, we shall use Fourier analysis. Recall that if $A$ is a subset of $\cG$ of density $\a$, then by Parseval's identity we have $\a=\sum_r|\widehat{\AA}(r)|^2$. Also, if $B$ is a multiset in $\cG$, then by Parseval's identity and the convolution law, $\langle \AA*\mu_B,\AA\rangle=\sum_r|\widehat{\AA}(r)|^2\widehat\mu_B(r)$ (for a multiset $B$, we define $\mu_B(x)=\frac{|G|}{|B|}B(x)$ where $B(x)$ is the multiplicity of $x$ in $B$). Thus, the condition that $A$ is $(B,\eta)$-closed can be rewritten as the inequality 
\[\sum_r|\widehat{\AA}(r)|^2\widehat\mu_B(r)\geq\eta\sum_r|\widehat{\AA}(r)|^2.\]
Another fact we shall use later is that if $W$ is a subspace of $\cG$, then $\widehat{\mu_W}(r)$ equals $\E_{w\in W}(-1)^{r.w}$, which is 1 if $r$ belongs to the orthogonal complement of $W$ and 0 otherwise.

\begin{lemma}\label{triangle}
	Let $G$ be an Abelian group, let $A\subset G$ be a finite subset, let $\eta_1,\eta_2>0$, and let $b_1,b_2\in G$. Suppose that $A$ is $(\{b_1\},1-\eta_1)$-closed and $(\{b_2\},1-\eta_2)$-closed in $G$. Then $A$ is $(\{b_1+b_2\},1-\eta_1-\eta_2)$-closed in $G$.
\end{lemma}

\begin{proof}
	Let $A_{\text{bad}}=\{a\in A: a+b_2\not \in A\}$. Then $|A_{\text{bad}}|\leq \eta_2|A|$, by hypothesis. So when $a\in A$ is chosen randomly, we have that 
	\[\mathbb{P}\lbrack a+b_1+b_2\not \in A \rbrack \leq \mathbb{P}\lbrack a+b_1\not\in A\rbrack+\mathbb{P}\lbrack a+b_1\in A_{\text{bad}}\rbrack\leq \eta_1+\eta_2.\]
	The result follows.
\end{proof}

We are now in a position to deduce Theorem \ref{mainresult} from Lemma \ref{mainlemma}. In the proof, and in the rest of this section, whenever a new function $g_i$ appears, we mean that there exists a function $g_i$ with the claimed property.

\begin{proof}[Proof of Theorem \ref{mainresult}]
	Let $\theta,\d>0$, $\cB'\s \cB$ with $|\cB'|\geq \d|\cB|$, and suppose that $\cA\s \cG$ is $(\cB',1-\e)$-closed, where $\e$ is to be specified. Let $\cB''=\{b\in \cB': \cA \text{ is } (\{b\},1-2\e) \text{-closed}\}$. Clearly, $|\cB''|\geq \frac{1}{2}|\cB'|$. Using Lemma \ref{mainlemma}, we can find a multiset $Q$ with elements chosen from $g_1(d)\cB''$ such that the set of arrays $r\in \cG$ with $r.q=0$ for at least $(1-g_2(d))|Q|$ choices $q\in Q$ (where $g_2(d)>0$) is contained in $\sum_{I\subset \lbrack d\rbrack, I\neq \emptyset} V_I\otimes \F_2^{I^c}$ for subspaces $V_I\s \F_2^{I}$ of dimension at most $g_3(d,\d)$. Then, by Lemma \ref{triangle}, $\cA$ is $(\{b\},1-2g_1(d)\e)$-closed for every $b\in Q$, since each such $b$ is an element of $g_1(d)\cB''$. In particular, $\cA$ is $(Q,1-2g_1(d)\e)$-closed, and therefore 
	\[\sum_r|\hcA(r)|^2\widehat\mu_Q(r)\geq(1-2g_1(d)\e)\sum_r|\hcA(r)|^2.\]
By Markov's inequality, it follows that 
	\[\sum_{r:\widehat\mu_Q(r)<1-2g_2(d)}|\hcA(r)|^2\leq\frac{g_1(d)\e}{g_2(d)}\sum_r|\hcA(r)|^2.\]
Choosing $\e=\e(d,\theta)>0$ to be at most $\theta g_2(d)/g_1(d)$, we therefore have 
\[\sum_{r:\widehat\mu_Q(r)\geq 1-2g_2(d)}|\hcA(r)|^2\geq(1-\theta)\sum_r|\hcA(r)|^2.\]
	Now if $\widehat\mu_Q(r)\geq 1-2g_2(d)$ then $r.q=0$ for at least $(1-g_2(d))|Q|$ choices $q\in Q$. Thus, we have 
	\begin{equation}
		\sum_{r\in T}|\hcA(r)|^2\geq(1-\theta)\sum_{r\in \cG}|\hcA(r)|^2 \label{eq1}
	\end{equation} where $T=\sum_{I\subset \lbrack d\rbrack, I\neq \emptyset} V_I\otimes \F_2^{I^c}$. Let $R=T^{\perp}=\bigcap_{I\subset \lbrack d\rbrack, I\neq \emptyset} V_I^{\perp}\otimes \F_2^{I^c}$. Because $\widehat{\mu_R}$ is the characteristic function of $T$, (\ref{eq1}) is equivalent to the inequality
	\[\sum_{r\in \cG}|\hcA(r)|^2\widehat{\mu_R}(r)\geq (1-\theta)\sum_{r\in \cG}|\hcA(r)|^2,\]
	which in physical space is the inequality 
	\[\langle\b1_{\cA}*\b1_{\cA},\mu_R\rangle\geq (1-\theta)\|\b1_{\cA}\|_2^2=(1-\theta)\a,\]
	where $\a$ is the density of $\cA$. Equivalently,
	\[\langle\mu_{\cA}*\mu_R,\b1_{\cA}\rangle\geq 1-\theta,\]
	which tells us that if a random element of $\cA$ is added to a random element of $R$, then the sum belongs to $\cA$ with probability at least $1-\theta$. The number of triples $(a_1,a_2,r)\in \cA\times \cA\times R$ with $a_1+a_2=r$ is therefore at least $(1-\theta)|\cA||R|$, and therefore, by averaging, there exists $a\in \cA$ such that $|(R-a)\cap \cA|\geq (1-\theta)|R|=(1-\theta)|R-a|$. But $R-a$ is $g_3(d,\d)$-simple, so we can take $k=g_3(d,\d)$ and $\cD=R-a$.
\end{proof}

It remains to prove Lemma \ref{mainlemma}.

\subsection{The proof of Lemma \ref{mainlemma} in the matrix case}

For the reader's convenience, in this subsection we give the proof of Lemma \ref{mainlemma} in the matrix case: that is, the case when $d=2$. Accordingly, in this subsection, $\cG$ will be the group $\F_2^{n_1}\otimes \F_2^{n_2}$ and $\cB$ will be the multiset that consists of all rank-1 matrices $u_1\otimes u_2$ with $u_1\in \F_2^{n_1}$ and $u_2\in \F_2^{n_2}$, with multiplicity. (As already remarked, the non-trivial multiplicity comes from when $u_1$ or $u_2$ is zero.)

\begin{lemma} \label{findstructure}
	Let $\cB'\subset \cB$ be such that $|\cB'|\geq \delta |\cB|$. Then there exist $k$ depending only on $\delta$, a subspace $U\subset \F_2^{n_1}$, and a subspace $V_u\subset\F_2^{n_2}$ for each $u\in U$, such that all these subspaces have codimension at most $k$, and such that every $u\otimes v$ with $u\in U,v\in V_u$ belongs to $16\cB'$. Moreover, we can insist that all $V_u$ have the same codimension.
\end{lemma}

\begin{proof}
	For each $u\in \F_2^{n_1}$, let $\cB_u'=\{v\in \F_2^{n_2}: u\otimes v\in \cB'\}$ and let $T=\{u\in \F_2^{n_1}: |\cB_u'|\geq\frac{\delta}{2}2^{n_2}\}$. By averaging, we have that $|T|\geq\frac{\delta}{2}2^{n_1}$.
	
	If $u\in T$, then $\cB_u'$ has density at least $\frac{\delta}{2}$ in $\F_2^{n_2}$, so by Bogolyubov's lemma (see for example Proposition 4.39 of \cite{taovu}) $4\cB_u'=2\cB_u'-2\cB_u'$ contains a subspace $W_u$ of codimension at most $k_1=k_1(\d)$. By Bogolyubov's lemma again, there is a subspace $U$ of codimension at most $k_1$ contained in $4T$. Let $u\in U$. Then pick $t_1,t_2,t_3,t_4\in T$ arbitrarily such that $u=t_1+t_2+t_3+t_4$ and set $V_u=\bigcap_{i=1}^4 W_{t_i}$. Note that $V_u$ has codimension at most $4k_1$ for every $u\in U$. If $v\in V_u$, then $t_i\otimes v\in 4\cB'$ for $i=1,\dots,4$, therefore $u\otimes v\in 16\cB'$. Thus, we can take $k=4k_1$.
	
	For the last assertion, note that we may replace $V_u$ with any subspace of it, and we still have $u\otimes V_u\s 16\cB'$.
\end{proof}

We shall now prove that we may take $Q=\bigcup_{u\in U} (u\otimes V_u)$ in Lemma \ref{mainlemma}.


\begin{lemma} \label{Qworks}
	There exists an absolute constant $\e>0$ with the following property. Let $U$ be a subspace of $\F_2^{n_1}$ and for each $u\in U$ let $V_u$ be a subspace of $\F_2^{n_2}$ such that all these subspaces have codimension at most $k$ and all the $V_u$ have the same codimension. Let $Q$ be the multiset $\bigcup_{u\in U}(u\otimes V_u)$. Then the set of arrays $r\in \cG$ with $r.q=0$ for at least $(1-\e)|Q|$ choices $q\in Q$ is contained in $W_{12}+W_1\otimes \F_2^{n_2}+\F_2^{n_1}\otimes W_2$ where $W_{12}\s \F_2^{n_1}\otimes \F_2^{n_2},W_1\s \F_2^{n_1}$ and $W_2\s \F_2^{n_2}$ are subspaces of dimension at most $f(k)$.
\end{lemma}

Before we prove this lemma, we need to establish the following result. In the statement, and in the rest of this section, we use the following convention for the multiplication of arrays $r\in \F_2^{n_1}\otimes \dots \otimes \F_2^{n_{a+b}}$ and $s\in \F_2^{n_1}\otimes \dots \otimes \F_2^{n_{a}}$. Define $rs\in \F_2^{n_{a+1}}\otimes \dots \otimes \F_2^{n_{a+b}}$ to be the array with $(rs)_{i_{a+1},\dots,i_{a+b}}=\sum_{j_1,\dots,j_a}r_{j_1,\dots,j_a,i_{a+1},\dots,i_{a+b}}s_{j_1,\dots,j_a}$. Note that in the case when $r$ is a matrix and $s$ is a vector, this is not quite the same as the standard convention since we sum over the first coordinate of the matrix instead of the second. If $r$ and $s$ are arrays of the same size, then we use the notation $r.s$ for the product of $r$ and $s$, since then it coincides with the obvious notion of dot product, and it is useful to think of it that way. 

\begin{lemma} \label{smallrank}
	Let $U$ be a subspace of $\F_2^{n_1}$ and for each $u\in U$ let $V_u$ be a subspace of $\F_2^{n_2}$ such that all these subspaces have codimension at most $k$ and all the $V_u$ have the same codimension. Let $Q=\bigcup_{u\in U}(u\otimes V_u)$ and for $m=2^{k+3}$, let $r_1,\dots,r_m\in \cG$ be such that for every $i\leq m$ there are 
	at least $\frac 34|Q|$ choices of $q\in Q$ such that $r_i.q=0$. Then there exist $i\neq j$ such that $(r_i-r_j)w=0$ for at least $\rho(k)2^{n_1}$ choices $w\in \F_2^{n_1}$, for some positive function $\rho(k)$. Thus, there exist $i\neq j$ such that $r_i-r_j$ has rank at most $l(k)$.
\end{lemma}

\begin{proof}[Proof of Lemma \ref{smallrank}]
	For each $i$, there are at least $|U|/2$ choices of $u\in U$ such that $r_i.(u\otimes v)=0$ for more than $|V_u|/2$ choices of $v\in V_u$. But $r_i.(u\otimes v)=(r_iu).v$, so this implies that $r_iu\in V_u^\perp$. 
	
	It follows that for at least $|U|/4$ choices of $u\in U$ we have $r_iu\in V_u^{\perp}$ for at least $m/4$ choices of $i\leq m$. Since $m/4>|V_u^{\perp}|$, for every such $u$ there exist $i\neq j$ with $r_iu=r_ju$. Thus, for some $i\neq j$, there are at least $\frac{|U|}{4m^2}$ choices of $u\in U$ with $r_iu=r_ju$. The result follows.
\end{proof}

Lemma \ref{smallrank} reduces Lemma \ref{Qworks} to the following.

\begin{lemma} \label{reduced}
	There exists an absolute constant $\e'>0$ with the following property. Let $U$ be a subspace of $\F_2^{n_1}$ and for each $u\in U$ let $V_u$ be a subspace of $\F_2^{n_2}$ such that all these subspaces have codimension at most $k$ and all the $V_u$ have the same codimension. Let $Q$ be the multiset $\bigcup_{u\in U}(u\otimes V_u)$. Then the set of arrays $r\in \cG$ of rank at most $l$ such that $r.q=0$ for at least $(1-\e')|Q|$ choices $q\in Q$ is contained in $W_1\otimes \F_2^{n_2}+\F_2^{n_1}\otimes W_2$ where $W_1\s \F_2^{n_1}$ and $W_2\s \F_2^{n_2}$ are subspaces of codimension at most $f(k,l)$.
\end{lemma}

Indeed, take $\e'$ as given by Lemma \ref{reduced} and let $\e=\min(\e'/2,1/4)$. We claim that $\e$ is suitable for Lemma \ref{Qworks}. By Lemma \ref{smallrank}, we can find $x_1,\dots,x_m\in \cG$ , with $m\leq 2^{k+3}$, as follows. For every $i$, $x_i.q=0$ for at least $(1-\e)|Q|$ choices $q\in Q$, and if $r.q=0$ for at least $(1-\e)|Q|$ choices $q\in Q$, then $r-x_i$ has rank at most $g_1(k)$ for some $i\leq m$. But then $(r-x_i).q=0$ for at least $(1-\e')|Q|$ choices $q\in Q$, and by Lemma \ref{reduced}, $r-x_i$ is contained in $W_1\otimes \F_2^{n_2}+\F_2^{n_1}\otimes W_2$ where $W_1,W_2$ have codimension at most $g_2(k)$ and do not depend on $r$. So we may take $W_{12}$ to be the span of all the $x_i$.

\begin{proof}[Proof of Lemma \ref{reduced}]
	We shall prove that we can take $\e'$ to be $1/8$, $W_1$ to be $U^{\perp}$ and $W_2$ to be a subspace that we define as follows. Let $X=\{x\in \F_2^{n_2}: x\in V_u^{\perp} \text{ for at least } \frac{|U|}{10\cdot 2^{l}} \text{ choices } u\in U\}$. Since $|V_u^{\perp}|\leq 2^k$ for every $u\in U$, we have $|X|\leq \frac{|U|\cdot 2^k}{{|U|}/{10\cdot 2^l}}=10\cdot 2^{k+l}$. Let $W_2=\text{span}(X)$.
	
	Now let $r$ be a matrix of rank at most $l$ such that $r.q=0$ for at least $7|Q|/8$ choices of $q\in Q$. Then $ru\in V_u^{\perp}$ for at least $3|U|/4$ choices of $u\in U$. Since $r$ has rank at most $l$, we have $r\in \mathbb{F}_2^{n_1}\otimes H_2(r)$ for some subspace $H_2(r)\s \F_2^{n_2}$ with $\dim(H_2(r))\leq l$. By definition, for any $v\not \in W_2$, the number of $u\in U$ with $v\in V_u^{\perp}$ is at most $\frac{|U|}{10\cdot 2^l}$. Since $|H_2(r)|\leq 2^l$ and $ru\in H_2(r)$ for every $u\in U$, the number of $u\in U$ with $ru\in V_u^{\perp}\setminus W_2$ is at most $\frac{|U|}{10}$. It follows that for more than $|U|/2$ choices $u\in U$ we have $ru\in W_2$. This in fact implies that $ru\in W_2$ for all $u\in U$, which implies that $r\in U^{\perp}\otimes \F_2^{n_2}+\F_2^{n_1}\otimes W_2$.
\end{proof}

\subsection{The proof of Lemma \ref{mainlemma} in the general case}

It is convenient to introduce a few definitions.

\begin{definition}
	Let $k$ be a positive integer and let $\e>0$. Let $Q$ be a multiset with elements chosen from $\cG$ (with arbitrary multiplicity). We say that $Q$ is $(k,\a)$-\emph{forcing} if the set of all arrays $r\in \cG$ with $r.q=0$ for at least $\a|Q|$ choices $q\in Q$ is contained in a set of the from $\sum_{I\subset \lbrack d\rbrack,I\neq \emptyset} V_I\otimes \F_2^{I^c}$ for some choice of subspaces $V_I\s \F_2^{I}$ of dimension at most $k$. 
	\end{definition}
Note that the conclusion of Lemma \ref{mainlemma} is that the multiset $Q$ is $(f_2(d,\d),1-f_3(d))$-forcing.
	
\begin{definition}
	We say that the hypothesis (H) holds for $d$ if the following is true. For every $\d>0$ there exist $f_1(d),f_2(d,\d)$ and $f_3(d)>0$ such that if $\cB'\subset \cB$ is such that $|\cB'|\geq \d|\cB|$, then there exists a multiset $Q$, with elements chosen from $f_1(d)\cB'$, that is $(f_2(d,\d),1-f_3(d))$-forcing.
\end{definition}

Hypothesis (H) is the one that really interests us, since if it holds for every $d$ then Lemma \ref{mainlemma} is proved. However, in order to get an induction to work we shall need a slight strengthening that says that we can ask for the elements of $Q$ to belong to a large subspace of a suitable form.

\begin{definition}	
	We say that the hypothesis (H') holds for $d$ if the following is true. For every $\d>0$ and every positive integer $l$ there exist $f_1(d),f_2(d,\d,l)$ and $f_3(d)>0$ such that if $\cB'\subset \cB$ is such that $|\cB'|\geq \d|\cB|$, and for every non-empty $I\subset \lbrack d\rbrack$, $L_I\s \F_2^I$ is a subspace of codimension at most $l$, then there exists a multiset $Q$, with elements chosen from $(f_1(d)\cB')\cap \bigcap_{I} (L_I\otimes \F_2^{I^c})$, that is $(f_2(d,\d,l),1-f_3(d))$-forcing.
\end{definition}

In what follows, we shall prove that if (H) holds for $d$, then so does (H'), and that if (H') holds for all $d'<d$, then (H) holds for $d$. This completes the proof since (H) holds for $d=1$. Indeed, if $\cB'\s \F_2^{n_1}$ has $|\cB'|\geq \d\cdot2^{n_1}$, then by Bogolyubov's lemma, $4\cB'=2\cB'-2\cB'$ contains a subspace $U\s \F_2^{n_1}$ of codimension at most $g(\d)$. If $x.u=0$ for over half the elements of $U$, then the set of $u\in U$ with $x.u=0$ is not contained in a proper subspace, so $x\in U^\perp$, which has dimension at most $g(\d)$. This implies that $U$ is $(g(\d),3/4)$-forcing. 

\medskip

The next few results are needed for technical reasons. The set introduced in the next definition behaves well under certain algebraic operations, such as intersecting with a dense subspace. It is a generalization of the set $\bigcup_{u\in U}(u\otimes V_u)$ used in the previous subsection.


\begin{definition} \label{ksystem}
Suppose that we build a collection of subspaces as follows. We begin with a subspace $U\subset\F_2^{n_1}$. Then for each $u_1\in U$ we take a subspace $U_{u_1}\subset\F_2^{n_2}$, for each $u_1\in U$ and $u_2\in U_{u_1}$ we take a subspace $U_{u_1,u_2}\subset\F_2^{n_2}$, and so on up to subspaces $U_{u_1,\dots,u_{d-1}}$. Now let $Q$ be the multiset that consists of all tensors $u_1\otimes\dots\otimes u_d$ such that $u_1\in U, u_2\in U_{u_1},\dots,u_d\in U_{u_1,\dots,u_{d-1}}$. If all the subspaces in the collection have codimension at most $l$, then we say that $Q$ is an \emph{$l$-system}.
\end{definition}

\begin{lemma} \label{intersect}
	Let $Q$ be an $l$-system and let $Q'$ be a $l'$-system. Then $Q\cap Q'$ contains an $(l+l')$-system.
\end{lemma}

\begin{proof}
	Let $Q$ have spaces as in Definition \ref{ksystem} and let $Q'$ have spaces $U'_{u'_1,\dots,u'_{k-1}}$.
	We define an $(l+l')$-system $P$ contained in $Q\cap Q'$ as follows. Let $V=U\cap U'$. Suppose we have defined $V_{v_1,\dots,v_{j-1}}$ for all $j\leq k$. Let $v_1\in V,v_2\in V_{v_1},\dots,v_{k-1}\in V_{v_1,\dots,v_{k-2}}$.  We let $V_{v_1\dots,v_{k-1}}=U_{v_1\dots,v_{k-1}}\cap U'_{v_1\dots,v_{k-1}}$. This is well-defined and has codimension at most $l+l'$ in $\F_2^{n_k}$. Let $P$ be the corresponding $(l+l')$-system.
\end{proof}

\begin{lemma} \label{findsystem}
	Let $\cB'\subset \cB$ be a multiset such that $|\cB'|\geq \delta |\cB|$. Then there exists an $f_1(d,\d)$-system with elements chosen from $f_2(d)\cB'$.
\end{lemma}

\begin{proof}
	The proof is by induction on $d$. The case $d=1$ is an easy application of Bogolyubov's lemma. Suppose that the lemma has been proved for all $d'<d$ and let $\cB'\subset \cB$ be a multiset such that $|\cB'|\geq \delta |\cB|$. Let $\cD$ be the multiset $\{v_2\otimes\dots\otimes v_d: v_2\in \F_2^{n_2},\dots,v_d\in \F_2^{n_d}\}$. For each $u\in \F_2^{n_1}$, let $\cB_u'=\{s\in \cD: u\otimes s\in \cB'\}$ and let $T=\{u\in \F_2^{n_1}: |\cB_u'|\geq\frac{\delta}{2}|\cD|\}$. By averaging, we have that $|T|\geq\frac{\delta}{2}2^{n_1}$. Now by the induction hypothesis, for every $t\in T$, there exists an $f_1(d-1,\d/2)$-system $P_t$ in $\F_2^{n_2}\otimes \dots \otimes \F_2^{n_d}$ (the definition is analogous to the definition of a system in $\F_2^{n_1}\otimes \dots \F_2^{n_d}$), which is contained in $f_2(d-1)\cB'_t$. By Bogolyubov's lemma, $4T=2T-2T$ contains a subspace $U\s \F_2^{n_1}$ of codimension at most $g_1(\d)$. For each $u\in U$, write $u=t_1+t_2+t_3+t_4$ arbitrarily, and let $Q_u=P_{t_1}\cap P_{t_2}\cap P_{t_3}\cap P_{t_4}$, which is a $4f_1(d-1,\d/2)$-system, by Lemma \ref{intersect}. Now $Q=\bigcup_{u\in U} (u\otimes Q_u)$ is an $f_1(d,\d)$-system provided that $f_1(d,\d)\geq \max(4f_1(d-1,\d/2),g_1(\d))$. Moreover, for any $u\in U,s\in Q_u$, we have $u\otimes s\in 4f_2(d-1)\cB'$, so $Q$ is suitable provided that $f_2(d)\geq 4f_2(d-1)$.
\end{proof}

The next lemma is the last ingredient needed to prove that (H') follows from (H).

\begin{lemma} \label{systemspace}
	Let $Q$ be a $k$-system, and for each non-empty $I\s \lbrack d\rbrack$ let $L_I\s \F_2^{I}$ be a subspace of codimension at most $l$. Let $T=\bigcap_I(L_I\otimes \F_2^{I^c})$. Then $Q\cap T$ contains an $f(d,k,l)$-system.
\end{lemma}

\begin{proof}
	Let the spaces of $Q$ be $U_{u_1,\dots,u_{j-1}}$. It suffices to prove that for every $1\leq j\leq d$, and every $u_1\in U,\ \dots,\ u_{j-1}\in U_{u_1,\dots,u_{j-2}}$, the codimension of $(u_1\otimes \dots\otimes u_{j-1}\otimes U_{u_1,\dots,u_{j-1}})\cap \bigcap_{I\subset \lbrack j\rbrack, j\in I}(L_I\otimes \F_2^{\lbrack j\rbrack \setminus I})$ in $u_1\otimes \dots\otimes u_{j-1}\otimes U_{u_1,\dots,u_{j-1}}$ is at most $g_1(d,l)$. Thus, it suffices to prove that for every $I\s \lbrack j\rbrack$ with $j\in I$, the codimension of $(u_1\otimes \dots\otimes u_{j-1}\otimes U_{u_1,\dots,u_{j-1}})\cap(L_I\otimes \F_2^{\lbrack j\rbrack \setminus I})$ in $u_1\otimes \dots\otimes u_{j-1}\otimes U_{u_1,\dots,u_{j-1}}$ is at most $g_2(l)$. But this is equivalent to the statement that $\big((\bigotimes_{i\in I\setminus \{j\}}u_i)\otimes U_{u_1,\dots,u_{j-1}}\big)\cap L_I$ has codimension at most $g_2(l)$ in $(\bigotimes_{i\in I\setminus \{j\}}u_i)\otimes U_{u_1,\dots,u_{j-1}}$, which clearly holds with $g_2(l)=l$.
\end{proof}

\begin{lemma} \label{htoh'}
	If (H) holds for $d$, then so does (H').
\end{lemma}

\begin{proof}
	Let $\cB'$ and $L_I$ be given as in (H'). By Lemma \ref{findsystem}, we can choose a $g_1(d,\d)$-system in $g_2(d)\cB'$. Thus, by Lemma \ref{systemspace}, $(g_2(d)\cB')\cap\bigcap_I(L_I\otimes \F_2^{I^c})$ contains a $g_3(d,\d,l)$-system $\cB''$. Note that $\cB''\s \cB$ with $|\cB''|\geq g_4(d,\d,l)|\cB|$, where $g_4(d,\d,l)>0$. By hypothesis (H) applied to $\cB''$ in place of $\cB'$, it follows that there is a multiset $Q$, with elements chosen from $g_5(d)\cB''\s g_5(d)g_2(d)\cB'$, which is $(g_6(d,\d,l),1-g_7(d))$-forcing. Now $Q\subset \bigcap_I(L_I\otimes \F_2^{I^c})$, finishing the proof.
\end{proof}

The next result generalizes Lemma \ref{smallrank}. To state it, we need to find the equivalent of low-rank matrices in the higher dimensional case. The definition we use is essentially the same as that of the \emph{partition rank} of a tensor (see for example \cite{lovett} for a discussion of this notion). Indeed, if a tensor has partition rank at most $k$, then it is $k$-degenerate (in the sense of the next definition), and conversely if a tensor is $k$-degenerate, then it has partition rank is at most $2^{d-1}k$.
The second author has shown that the partition rank is also related to the \emph{analytic rank} of a tensor \cite{janzer}, which we do not define here (but again see \cite{lovett}), with a tower-type dependence, improving on the previously known Ackermann dependence. In this section we use a very similar argument, but since we do not care about bounds, we present it in qualitative form for ease of reading and for the sake of completeness. 

\begin{definition}
	Let $k$ be a positive integer. We say that $r\in \cG$ is $k$-\emph{degenerate} if there is a collection of subspaces $H_I\subset\F_2^I$, one for each non-empty proper subset $I\s \lbrack d\rbrack$ and each one of dimension at most $k$, such that $r\in \sum_{I\subset \lbrack d-1 \rbrack,I\neq \emptyset}H_I\otimes H_{I^c}$.
\end{definition}

If $r\in H_I\otimes \F_2^{I^c}$ with $\dim(H_I)\leq k$, then $r\in H_I\otimes H_{I^c}$ for some $H_{I^c}\s \F_2^{I^c}$ of dimension at most $k$. (This follows by writing $r$ as $\sum_{j\leq m} s_j\otimes t_j$ with $\{s_j\}$ a basis for $H_I$ and letting $H_{I^c}$ be the span of all $t_j$.) Thus, $r$ is $k$-degenerate if and only if we have the apparently weaker condition that $r\in \sum_{I\s \lbrack d-1 \rbrack, I\neq \emptyset} H_I\otimes \F_2^{I^c}$ for some collection of subspaces $H_I\s \F_2^{I}$ of dimension at most $k$, or equivalently the condition that $r\in \sum_{I\s \lbrack d-1 \rbrack, I\neq \emptyset} \F_2^{I}\otimes H_{I^c}$ for some collection of subspaces $H_{I^c}\s \F_2^{I^c}$ of dimension at most $k$.


\begin{lemma} \label{focusondeg}
	Suppose that (H) holds for $d-1$. Let $\cD$ be the multiset $\{u_1\otimes \dots \otimes u_{d-1}:u_1\in \F_2^{n_1},\dots,u_{d-1}\in\F_2^{n_{d-1}}\}$.
	
	\begin{enumerate}[label=(\roman*)]
		\item Let $\cD'\s \cD$, and for each $t\in \cD'$, let $U_t$ be a subspace of $\F_2^{n_d}$ of codimension at most $k$ with the codimensions of the $U_t$ all equal. Let $Q$ be the multiset $\bigcup_{t\in \cD'}(t\otimes U_t)$. If $m=2^{k+3}$, and $r_1,\dots,r_m\in \cG$ have the property that for every $i$, $r_i.q=0$ holds for at least $\frac{3}{4}|Q|$ choices of $q\in Q$, then there exist $i\neq j$ such that $(r_i-r_j)(v_1\otimes \dots \otimes v_{d-1})=0$ for at least $f_1(k)|\cD'|$ choices of $v_1\otimes \dots \otimes v_{d-1}\in \cD'$, where $f_1(k)>0$ depends only on $k$.
		
		\item Let $c>0$ and let $r\in \cG$ be such that $r(v_1\otimes \dots \otimes v_{d-1})=0$ for at least $c|\cD|$ choices $v_1\otimes \dots\otimes v_{d-1}\in \cD$. Then $r$ is $f_2(d,c)$-degenerate.
		
		\item Let $\d>0$ and let $\cB'\s \cB$ be such that $|\cB'|\geq \d |\cB|$. Then there exists $Q\subset 4\cB'$ with the following property. If $m\geq f_3(\d)$ and $r_1,\dots,r_m\in \cG$ are such that for each $i$, $r_i.q=0$ for at least $\frac{3}{4}|Q|$ choices $q\in Q$, then there exist $i\neq j$ such that $r_i-r_j$ is $f_4(d,\d)$-degenerate.
		
		\item Let $\d>0$ and let $\cB'\s \cB$ be such that $|\cB'|\geq \d |\cB|$. Then there exist $Q\subset 4\cB'$ and a subspace $V_{\lbrack d\rbrack}\s \F_2^{\lbrack d\rbrack}$ of dimension at most $f_3(\d)$ with the following property. Any array $r$ with $r.q=0$ for at least $\frac{3}{4}|Q|$ choices $q\in Q$ can be written as $r=x+y$ where $x\in V_{\lbrack d \rbrack}$ and $y$ is $f_4(d,\d)$-degenerate.
	\end{enumerate}
\end{lemma}

\begin{proof}
	\begin{enumerate}[label=(\roman*)]
		\item By averaging, for every $i\leq m$ there are at least $|\cD'|/2$ choices of $t\in\cD'$ such that $r_i.(t\otimes u)=0$ for more than half the $u\in U_t$. But $r_i.(t\otimes u)=(r_it).u$, so for these $t$ we have that $r_it\in U_t^\perp$. By further averaging, it follows that for at least $|\cD'|/4$ choices of $t\in \cD'$ we have that $r_it\in U_t^{\perp}$ for at least $m/4$ choices of $i\leq m$. Since $m/4>|U_t^{\perp}|$, for every such $t$, there exist $i\neq j$ with $r_it=r_jt$. Thus, for some $i\neq j$, there are at least $\frac{|\cD'|}{4m^2}$ choices of $t\in \cD'$ with $r_it=r_jt$, which implies the statement we wish to prove.
		\item Write $r=\sum_{i} s_i\otimes w_i$ where $s_i\in \F_2^{n_1}\otimes \dots \otimes \F_2^{n_{d-1}}$ and $\{w_i\}$ is a basis for $\F_2^{n_d}$. Note that $r(v_1\otimes \dots \otimes v_{d-1})=0$ if and only if $s_i.(v_1\otimes \dots \otimes v_{d-1})=0$ for every $i$. Let $\cD'$ be the multiset consisting of all $v_1\otimes \dots \otimes v_{d-1}$ with $s_i.(v_1 \otimes \dots \otimes v_{d-1})=0$ for every $i$. By assumption, $|\cD'|\geq c|\cD|$. Since (H) holds for $d-1$, there exists a multiset $Q$, with elements chosen from $g_1(d)\cD'$, which is $(g_2(d,c),1-g_3(d))$-forcing. Since $s_i.q=0$ for every $q\in Q$, it follows that there exist subspaces $V_I\s \F_2^{I}$ for every $I\s \lbrack d-1\rbrack,I\neq \emptyset$, of dimension at most $g_2(d,c)$ such that $s_i\in \sum_{I}V_I\otimes \F_2^{\lbrack d-1\rbrack \setminus I}$ for all $i$. In particular, $r\in \sum_{I\s \lbrack d-1 \rbrack, I\neq \emptyset} V_I\otimes \F_2^{I^c}$. Thus, $r$ is $g_2(d,c)$-degenerate.
		\item Let $\cD'=\{t\in \cD: t\otimes u\in \cB' \text{ for at least } \frac{\d}{2}2^{n_d} \text{ choices } u\in \F_2^{n_d}\}$. Clearly, we have $|\cD'|\geq \frac{\d}{2}|\cD|$. Moreover, by Bogolyubov's lemma, for every $t\in \cD'$, there exists a subspace $U_t\s \F_2^{n_d}$ of codimension at most $g_4(\d)$ such that $t\otimes U_t\s 4\cB'$. After passing to suitable subspaces, we may assume that all $U_t$ have the same codimension. Now let $Q=\bigcup_{t\in \cD'}(t\otimes U_t)$. By (i), if $f_3(\d)\geq 2^{g_4(\d)+3}$, then there exist $i\neq j$ such that $(r_i-r_j)t=0$ for at least $g_5(\d)|\cD'|$ choices $t\in \cD'$ (where $g_5(\d)>0$) and therefore also for at least $g_6(\d)|\cD|$ choices $t\in \cD$ (where $g_6(\d)>0$). By (ii), it follows that $r_i-r_j$ is $g_7(d,\d)$-degenerate.
		\item Choose $Q$ as in (iii) and let $r_1,\dots,r_m$ be a maximal set such that for all $i$ we have $r_i.q=0$ for at least $\frac{3}{4}|Q|$ choices $q\in Q$, and for all $i\neq j$, $r_i-r_j$ is not $f_4(d,\d)$-degenerate. Then, by (iii), we have $m< f_3(\d)$. Let $V_{\lbrack d\rbrack}$ be the span of all the $r_i$. The result follows by the maximality of $\{r_1,\dots,r_m\}$.
	\end{enumerate}
\end{proof}

The next lemma is the key step to complete the proof of Lemma \ref{mainlemma}. In order to state it, we need to introduce a total ordering $\prec$ of the non-empty subsets of $\lbrack d-1\rbrack$. It does not matter exactly what the ordering is, but we require it to have the property that if $J\subsetneq I$ then $J\prec I$. 

To understand the point of the next lemma, observe that we take an array $r$ that belongs to the sum of a certain set of subspaces, some of which depend on $r$, and we show, using the hypothesis that $r.q=0$ for almost all $q\in Q$, that it is contained in a similar sum, but with the subspace corresponding to $I^c$ no longer depending on $r$. By applying this lemma repeatedly, we shall remove all dependence on $r$ from the right-hand side.

\begin{lemma} \label{keylemma}
	Suppose that (H') holds for every $d'<d$. Let $\d>0$ and let $\cB'\s \cB$ be such that $|\cB'|\geq \d|\cB|$. Let $I\s \lbrack d-1\rbrack, I\neq \emptyset$ and for each $J\prec I$ let $W_J\s \F_2^{J},W_{J^c}\s \F_2^{J^c}$ be subspaces of dimension at most $k$. In addition, let $W_{\lbrack d\rbrack}\s \F_2^{\lbrack d\rbrack}$ be a subspace of dimension at most $k$. Then there exists a multiset $Q$ with elements chosen from $f_1(d)\cB'$ with the following property. Any array 
	\[r\in W_{\lbrack d\rbrack}+ \sum_{J\prec I} (W_J\otimes \F_2^{J^c}+\F_2^{J}\otimes W_{J^c})+\sum_{J\succeq I} \F_2^{J}\otimes H_{J^c}(r)\] 
	with $\dim(H_{J^c}(r))\leq k$ and the property that $r.q=0$ for at least $(1-f_2(d))|Q|$ choices $q\in Q$ is contained in a subspace 
	\[W_{\lbrack d \rbrack}+\sum_{J\preceq I} (U_J\otimes \F_2^{J^c}+\F_2^{J}\otimes U_{J^c})+\sum_{J\succ I} \F_2^{J}\otimes K_{J^c}(r)\] 
	for some $U_J\s \F_2^{J},U_{J^c}\s \F_2^{J^c}$ not depending on $r$ and some $K_{J^c}(r)\s \F_2^{J^c}$ possibly depending on $r$, all of dimension at most $f_3(d,\d,k)$.
\end{lemma}

\begin{proof}
	After reordering if necessary, we may assume that $I=\lbrack a\rbrack$ for some $1\leq a \leq d-1$. The next claim gives us a multiset $Q$ with certain properties. Once we have it, we shall use those properties to show that $Q$ satisfies the conclusion of the lemma.
	\bigskip
	
	\emph{Claim.} There exist a multiset $Q'$ with elements chosen from $\bigcap_{J\subset I,J\neq I,J\neq \emptyset}(W_J^{\perp}\otimes \F_2^{I \setminus J})$ which is $(h_1(d,\d,k),1-h_2(d))$-forcing (with $h_2(d)>0$), and for each $s\in Q'$, a multiset $Q_s$ with elements chosen from $\F_2^{I^c}$ which is $(h_3(d,\d),1-h_4(d))$-forcing (with $h_4(d)>0$) such that $\max_{s\in Q'}|Q_s|\leq 2\min_{s\in Q'}|Q_s|$ and setting $Q$ to be the multiset $\{s\otimes t: s\in Q',t\in Q_s\}=\cup_{s\in Q'}(s\otimes Q_s)$, the elements of $Q$ belong to $h_5(d)\cB'$
	
	\emph{Proof of Claim.} Let $\cC$ be the multiset $\{u_1\otimes \dots \otimes u_a: u_i\in \F_2^{n_i}\}$ and let $\cD$ be the multiset $\{u_{a+1}\otimes \dots \otimes u_d: u_i\in \F_2^{n_i}\}$. For each $s\in \cC$, let $\cD_s=\{t\in \cD: s\otimes t\in \cB'\}$. Also, let $\cC'=\{s\in \cC: |\cD_s|\geq \frac{\d}{2}|\cD|\}$. Clearly, $|\cC'|\geq \frac{\d}{2}|\cC|$. By Lemma \ref{findsystem}, for every $s\in \cC'$ there exists a $g_1(d,\d)$-system $R_s$ in $g_2(d)\cD_s$. By (H') applied to $a$, there exists a multiset $Q'$ with elements chosen from $(g_3(d)\cC')\cap \bigcap_{J\subset I,J\neq I,J\neq \emptyset}(W_J^{\perp}\otimes \F_2^{I \setminus J})$ which is $(g_4(d,\d,k),1-g_5(d))$-forcing for some $g_5(d)>0$. For every $s\in Q'$, choose $s_1,\dots,s_l\in \cC'$ with $l\leq g_3(d)$ such that $s=s_1+\dots+s_l$, and let $P_s=\bigcap_{i\leq l}R_s$. By Lemma \ref{intersect}, $P_s$ contains a $g_6(d,\d)$-system, therefore $|P_s|\geq g_7(d,\d)|\cD|$ for some $g_7(d,\d)>0$. By (H) for $d-a$, for every $s\in Q'$ there exists a multiset $Q_s$ with elements chosen from $g_8(d)P_s$ which is $(g_9(d,\d),1-g_{10}(d))$-forcing for some $g_{10}(d)>0$. Notice that if we repeat every element of $Q_s$ the same number of times, then the multiset we obtain is still $(g_9(d,\d),1-g_{10}(d))$-forcing, so we may assume that $\max_{s\in Q'}|Q_s|\leq 2\min_{s\in Q'}|Q_s|$. Define $Q=\{s\otimes t:s\in Q',t\in Q_s\}=\bigcup_{s\in Q'}(s\otimes Q_s)$. Note that as $R_s\subset g_2(d)\cD_s$ for all $s\in \cC'$, we have $s\otimes R_s\subset g_2(d)\cB'$ for all $s\in \cC'$. But the elements of $Q'$ are chosen from $g_3(d)\cC'$, so $s\otimes P_s\s g_3(d)g_2(d)\cB'$ for all $s\in Q'$. Finally, the elements of $Q_s$ are chosen from $g_8(d)P_s$, so the elements of $s\otimes Q_s$ are chosen from $g_8(d)g_3(d)g_2(d)\cB'$ for every $s\in Q'$. This completes the proof of the claim.
	
	\medskip
	
	Now let us show that $Q$ does indeed satisfy the conclusion of the lemma. Since for every $s\in Q'$, $Q_s$ is $(h_3(d,\d),1-h_4(d))$-forcing, there exist subspaces $V_J(s)\s \F_2^{J}$ for every $J\s I^c,J\neq \emptyset$, with dimension at most $h_3(d,\d)$ such that the set of arrays $t\in \F_2^{I^c}$ with $t.q=0$ for at least $(1-h_4(d))|Q_s|$ choices $q\in Q_s$ is contained in $\sum_{J\s I^c,J\neq \emptyset} V_J(s)\otimes \F_2^{I^c\setminus J}$. Let $f_2(d)=\frac{1}{4}h_2(d)h_4(d)$. Now if $r\in \cG$ has $r.q=0$ for at least $(1-f_2(d))|Q|$ choices $q\in Q$, then by averaging, for at least $(1-\frac{1}{2}h_2(d))|Q'|$ choices $s\in Q'$, we have $r.(s\otimes t)=0$ for at least $(1-h_4(d))|Q_s|$ choices $t\in Q_s$. Therefore (noting that $r.(s\otimes t)=(rs).t$), $rs\in \sum_{J\s I^c,J\neq \emptyset} V_J(s)\otimes \F_2^{I^c\setminus J}$ for at least $(1-\frac{1}{2}h_2(d))|Q'|$ choices $s\in Q'$.
	
	Now let $g_{11}(d,k)=\frac{h_2(d)}{2\cdot 2^{k}}$ and $g_{12}(d,\d,k)=h_3(d,\d)+2^{d+1}k$. Let $X$ be the subset of $\F_2^{I^c}$ consisting of those arrays $x$ such that for at least $g_{11}(d,k)|Q'|$ choices $s\in Q'$, there exists some $t(s)\in V_{I^c}(s)+\sum_{J\subset I,J\neq I}((\F_2^{J}\otimes W_{J^c})s)$ for which $x-t(s)$ is $g_{12}(d,\d,k)$-degenerate. (Here and below, for a subspace $L\s \cG$ and an array $s\in \F_2^{I}$, we write $Ls$ for the subspace $\{rs: r\in L\}\s \F_2^{I^c}$.) Choose a maximal subset $\{x_1,\dots,x_m\}\s X$ such that no $x_i-x_j$ with $i\neq j$ is $2g_{12}(d,\d,k)$-degenerate. Then there do not exist $i\neq j$, $s\in Q'$ and $t\in V_{I^c}(s)+\sum_{J\subset I,J\neq I} ((\F_2^{J}\otimes W_{J^c})s)$ with $x_i-t$ and $x_j-t$ both $g_{12}(d,\d,k)$-degenerate. It follows, by the definition of $X$ and the fact that the dimension of $V_{I^c}(s)+\sum_{J\subset I,J\neq I} ((\F_2^{J}\otimes W_{J^c})s)$ is at most $h_3(d,\d)+2^dk$, that $mg_{11}(d,k)|Q'|\leq |Q'|\cdot 2^{h_3(d,\d)+2^dk}$, and therefore that $m\leq g_{13}(d,\d,k)$. Thus, there exists a set $Y\s X$ of size at most $g_{13}(d,\d,k)$ such that for any $x\in X$, there exists some $y\in Y$ such that $x-y$ is $2g_{12}(d,\d,k)$-degenerate. Let $Z=\text{span}(Y)$. Then $\dim(Z)\leq g_{13}(d,\d,k)$, and for every $x\in X$ there is some $z\in Z$ such that $x-z$ is $2g_{12}(d,\d,k)$-degenerate.
	
	Let $r\in W_{\lbrack d\rbrack}+\sum_{J\prec I} (W_J\otimes \F_2^{J^c}+\F_2^{J}\otimes W_{J^c})+\sum_{J\succeq I} \F_2^{J}\otimes H_{J^c}(r)$ with $\dim(H_{J^c}(r))\leq k$ and the property that $r.q=0$ for at least $(1-f_2(d))|Q|$ choices $q\in Q$. Let $Q'(r)$ be the submultiset of $Q'$ consisting of those $s\in Q'$ for which $rs\in \sum_{J\s I^c,J\neq \emptyset} V_J(s)\otimes \F_2^{I^c\setminus J}$. As we have seen two paragraphs above, $|Q'(r)|\geq (1-\frac{1}{2}h_2(d))|Q'|$.
	
	Note that we can write $r=r_1+r_2+r_3+r_4$ where $r_1\in \sum_{J\subset I,J\neq I,J\neq \emptyset} W_J\otimes \F_2^{J^c}$, $r_2\in \sum_{J\prec I, J\not \subset I} (W_J\otimes \F_2^{J^c}+\F_2^{J}\otimes W_{J^c})+\sum_{J\succ I} \F_2^{J}\otimes H_{J^c}(r)$, $r_3\in W_{\lbrack d\rbrack}+\sum_{J\subset I,J\neq I,J\neq \emptyset} \F_2^{J}\otimes W_{J^c}$ and $r_4\in \F_2^{I}\otimes H_{I^c}(r)$. Since the elements of $Q'$ belong to $\bigcap_{J\subset I,J\neq I,J\neq \emptyset}(W_J^{\perp}\otimes \F_2^{I \setminus J})$, we have $r_1s=0$ for every $s\in Q'$. Since $\dim(W_J),\dim(W_{J^c}),\dim(H_{J^c}(r))\leq k$, $r_2s$ is $2^{d+1}k$-degenerate. Also, $r_3s\in \sum_{J\subset I,J\neq I} ((\F_2^{J}\otimes W_{J^c})s)$. It follows that for every $s\in Q'(r)$, there exists some $t(s)\in V_{I^c}(s)+\sum_{J\subset I,J\neq I} ((\F_2^{J}\otimes W_{J^c})s)$ such that $r_4s-t(s)$ is $g_{12}(d,\d,k)=(h_3(d,\d)+2^{d+1}k)$-degenerate (we have used that $\dim (V_{J}(s))\leq h_3(d,\d)$). In particular, if $t\not \in X$, then the number of choices $s\in Q'(r)$ for which $r_4s=t$ is at most $g_{11}(d,\d,k)|Q'|$. On the other hand, notice that $r_4s\in H_{I^c}(r)$ for every $s\in Q'$. Since $|H_{I^c}(r)|\leq 2^{k}$, it follows that $r_4s\in X$ for at least $|Q'(r)|-2^kg_{11}(d,\d,k)|Q'|=|Q'(r)|-\frac{1}{2}h_2(d)|Q'|\geq (1-h_2(d))|Q'|$ choices $s\in Q'$.
	
	Let $X(r)=X\cap H_{I^c}(r)$. Let $t_1,\dots,t_{\a}$ be a maximal linearly independent subset of $X(r)$ and extend it to a basis $t_1,\dots,t_{\a},t'_1,\dots,t'_{\beta}$ for $H_{I^c}(r)$. Now if a linear combination of $t_1,\dots,t_{\a},t'_1,\dots,t'_{\beta}$ is in $X$, then the coefficients of $t'_1,\dots,t'_{\beta}$ are all zero. Write $r_4=\sum_{i\leq \a}s_i\otimes t_i+\sum_{j\leq \beta}s'_j\otimes t'_j$ for some $s_i,s'_j\in \F_2^{I}$. Since $r_4q\in X$ for at least $(1-h_2(d))|Q'|$ choices $q\in Q'$, we have, for all $j$, that $s'_j.q=0$ for at least $(1-h_2(d))|Q'|$ choices $q\in Q'$. Thus, as $Q'$ is $(h_1(d,\d,k),1-h_2(d))$-forcing, there exist subspaces $L_J\s \F_2^{J}$ ($J\s I,J\neq \emptyset$), not depending on $r$ and of dimension at most $h_1(d,\d,k)$, such that $s'_j\in \sum_{J\s I,J\neq \emptyset} L_J\otimes \F_2^{I\setminus J}$ for all $j$. Thus, $r_4\in \sum_{i\leq \a}s_i\otimes t_i+\sum_{J\s I,J\neq \emptyset} L_J\otimes \F_2^{J^c}$. Moreover, for every $i\leq \a$, we have $t_i\in X$, so there exist $z_i\in Z$ such that $t_i-z_i$ is $2g_{12}(d,\d,k)$-degenerate. It follows that $r_4\in \F_2^{I}\otimes Z+\sum_{J\supset I,J\neq I,J\subset \lbrack d-1\rbrack} \F_2^{J}\otimes K'_{J^c}(r)+\sum_{J\s I,J\neq \emptyset} L_J\otimes \F_2^{J^c}$ for some $K'_{J^c}(r)\s \F_2^{J^c}$ of dimension at most $\a\cdot 2g_{12}(d,\d,k)\leq 2^k\cdot 2g_{12}(d,\d,k)$. This finishes the proof of the lemma.
\end{proof}

We are now in a position to complete the proof of Lemma \ref{mainlemma}.

\begin{proof}[Proof of Lemma \ref{mainlemma}]
	As remarked earlier, it suffices to prove that (H) holds for all $d$. But (H) implies (H'), and (H) holds for $d=1$, therefore it suffices to prove, assuming that (H') holds for all $d'<d$, that (H) holds for $d$.
	
	Let $\cB'$ given as in the definition of (H). By Lemma \ref{focusondeg} (iv), there exist $Q_{\emptyset}\s 4\cB'$ and $V_{\lbrack d\rbrack}\s \F_2^{\lbrack d\rbrack}$ of dimension at most $g_1(\d)$ such that if $r.q=0$ for at least $\frac{3}{4}|Q_{\emptyset}|$ choices $q\in Q_{\emptyset}$, then $r$ can be written as $r=x+y$, where $x\in V_{\lbrack d\rbrack}$ and $y$ is $g_2(d,\d)$-degenerate. By repeated application of Lemma \ref{keylemma}, for all $I\s \lbrack d-1\rbrack$ with $I\neq \emptyset$, in the order given by $\prec$, we obtain for each such $I$ a multiset $Q_{I}$ with elements chosen from $g_3(d)\cB'$, subspaces $V_I\s \F_2^{I}$, $V_{I^c}\s \F_2^{I^c}$ of dimension at most $g_4(d,\d)$, and we also obtain a constant $g_5(d)>0$ such that the following statement holds. If $r\in \cG$ is such that for each $I\subset[d-1]$ we have that $r.q=0$ for at least $(1-g_5(d))|Q_I|$ choices of $q\in Q_I$, then $r\in V_{\lbrack d\rbrack}+\sum_{I\s \lbrack d-1\rbrack,I\neq \emptyset} (V_I\otimes \F_2^{I^c}+\F_2^{I}\otimes V_{I^c})$. Moreover, after taking several copies of each $Q_I$, we may assume that $\max_I |Q_I|\leq 2\min_{I} |Q_I|$. Now let $g_6(d)=\frac{g_5(d)}{2\cdot 2^{d-1}}$, let $Q=\bigcup_{I}Q_I$, and suppose that $r.q=0$ for at least $(1-g_6(d))|Q|$ choices $q\in Q$. Then, by averaging, for every $I\s \lbrack d-1\rbrack$ we have $r.q=0$ for at least $(1-g_5(d))|Q_I|$ choices of $q\in Q_I$. Thus, $Q$ is $(g_7(d,\d),1-g_6(d))$-forcing, where $g_7(d,\d)=\max(g_1(\d),g_4(d,\d))$.
\end{proof}

\section{The counterexample to Question \ref{mainq}}

We shall now present an example that gives a negative answer to Question \ref{mainq}. The example is easy to define, but it takes a little work to prove that it has the properties we require. In what follows, let $G=\F_2^n$. For a vector $v\in G$ write $|v|$ for the number of entries equal to 1 in $v$. Then our set $A$ will be $\{v\in \F_2^n:|v|\leq n/2-10^{20}n^{3/4}\}$, and our set $B$ will be $\{v\in \F_2^n:|v|=n^{1/2}\}$. 



Note first that $A$ is $\eta$-closed with respect to $B$ where $\eta>0$ is some absolute constant. Indeed, by the central limit theorem, when $n$ sufficiently large, the probability that a random element $u\in A$ has $|u|\leq n/2-10^{20}n^{3/4}-n^{1/4}$ is at least some absolute constant $\eta_1$, and conditional on this, the probability that $|u+v|\in A$ for a random element $v\in B$ is at least some other absolute constant $\eta_2$, so we may take $\eta=\eta_1 \eta_2$. What we shall prove is that for this $\eta$, with $\e=0.99$, say, there do not exist $c,\d$ and $l$ with the properties described in Question \ref{mainq}. In fact, we shall prove the slightly stronger statement that for any $\delta>0$ and positive integer $l$, if $n$ is sufficiently large then there do not exist $C\subset A+lB$ and $B'\subset B$ with $|B'|\geq \d |B|$ such that $C$ is $(B',0.99)$-closed. Since for sufficiently large $n$, we have $A+lB\subset A'=\{v\in \F_2^n:|v|\leq n/2-10^{15}n^{3/4}\}$, it suffices to prove the same statement but with $A+lB$ replaced by $A'$. From now on, we always assume that $n$ is sufficiently large.

The proof relies on two lemmas and a definition.

\begin{lemma} \label{fourier}
	If $B'\subset B$ has $|B'|\geq \delta |B|$, then $\widehat{\mu_{B'}}(u)\geq 0.98$ for at most $\exp(n^{2/3})$ vectors $u\in \F_2^n$.
\end{lemma}

\begin{definition}
	Given $B'\subset B$, we say $u\in A'$ is \emph{$B'$-compatible} if the number of $w\in B'$ with $|u+w|\leq |u|$ is at least $|B'|/3$.
\end{definition}

\begin{lemma} \label{compatible}
	Let $B'\subset B$ have $|B'|\geq\d|B|$. Then the number of those $u\in A'$ which are $B'$-compatible is at most $\exp(-n^{3/4})|A'|$.
\end{lemma}

\smallskip

\noindent Let us see why these two lemmas are sufficient. Suppose that $C\subset A'$ is $(B',0.99)$-closed for some $B'\subset B$ with $|B'|\geq \d|B|$. Let $w\in B'$ be chosen at random. Then the expected number of $u\in C$ such that $u+w\in C$ is at least $0.99|C|$, so by considering all such pairs $\{u,u+w\}$ and noting that $(u+w)+w=u$, we see that there are on average at least $\frac{0.99}2|C|$ choices of $u\in C$ such that $|u+w|\leq|u|$. Therefore, if $u\in C$ is chosen at random, the average number of $w\in B'$ such that $|u+w|\leq|u|$ is at least $\frac{0.99}2|B'|$. It follows that for at least $|C|/10$ elements of $C$ the number of such $w$ is at least $|B'|/3$, so at least $|C|/10$ elements of $A'$ are $B'$-compatible. Lemma \ref{compatible} then implies that $C$ has density at most $10\exp(-n^{3/4})$ in $G$. Let us write $\g$ for this density.

On the other hand, since $C$ is $(B',0.99)$-closed, we have the inequality
\[\sum_{u\in G} \widehat{\mu_{B'}}(u)|\hC(u)|^2\geq 0.99\sum_{u\in G}|\hC(u)|^2,\]
which implies that 
\[\sum_{u\in G: \widehat{\mu_{B'}}(u)\geq 0.98} \widehat{\mu_{B'}}(u)|\hC(u)|^2\geq 0.01\sum_{u\in G}|\hC(u)|^2.\]
Using Lemma \ref{fourier}, together with the observations that $\widehat{\mu_{B'}}(u)\leq 1$ and $|\hC(u)|\leq \g$ for every $u\in G$ and that $\sum_{u\in G}|\hC(u)|^2=\g$, we deduce that $\exp(n^{2/3})\g^2\geq 0.01 \g$, so $\g\geq 0.01\exp(-n^{2/3})$. For sufficiently large $n$, this contradicts the upper bound for $\g$ that we obtained a few lines above. 
\medskip

It remains to prove the two lemmas. The next two results are preparation for the proof of Lemma \ref{fourier}.

\begin{lemma}\label{Claim2} Let $V$ be a subspace of $\mathbb{F}_2^n$ of dimension $d$ such that every $v\in V$ has $|v|\geq n^{8/15}$. Then $V$ has a basis $\{v_1,\dots,v_d\}$ such that for every $i$, the set $I_i=\{k\leq n: v_i(k)=1,v_j(k)=0 \text{ for all } j\neq i\}$ has size at least $n^{8/15}/2^{d-1}$. Here and below, the $k$th entry of a vector $v$ is denoted by $v(k)$.
\end{lemma}

\begin{proof} We use induction on $d$. The case $d=1$ easily follows from the assumption on $V$. Let $V'$ have dimension $d+1$ and suppose that for a $d$-dimensional subspace $V\subset V'$, $v_1,\dots,v_d$ and $I_1,\dots,I_d$ have been chosen satisfying the requirements. Choose some $v\in V'\setminus V$. Replacing $v$ by $v-v_1$ if necessary, we may assume that $v(k)=0$ for at least $|I_1|/2$ choices $k\in I_1$. Similarly, we may assume that $v(k)=0$ for at least $|I_i|/2$ choices $k\in I_i$ for every $i\leq d$. Thus, there exist subsets $J_1,\dots,J_d$ of $\{1,\dots,n\}$ of size at least $n^{8/15}/2^d$ each such that for every $i\leq d$ and every $k\in J_i$ we have $v_i(k)=1$ but  $v_j(k)=0$ for all $j$ with $j\neq i,j\leq d$, and $v(k)=0$. Let $J=\{k\leq n: v(k)=1\}$. By the assumption on $V'$, we have $|J|\geq n^{8/15}$. Now it is easy to see that we can define $v_1'$ to be $v_1$ or $v_1-v$ and achieve that $v_1'(k)=0$ for at least $|J|/2$ choices of $k\in J$. Similarly, we can define $v_2',\dots,v_d'$ such that each $v_i'$ is $v_i$ or $v_i-v$ and $v_1'(k)=\dots=v_d'(k)=0$ for at least $|J|/2^d$ choices $k\in J$. Then for any $i,j\leq d$, we have $v_i'(k)=v_i(k)$ for every $k\in J_j$, and it follows that for any $i\leq d$ and $k\in J_i$, we have $v_i'(k)=1$ but $v_j'(k)=0$ for all $j\neq i$, and $v(k)=0$. Thus, the set $\{v_1',\dots,v_d',v\}$ is suitable so the lemma is proved.
\end{proof}


\begin{corollary}\label{Claim1} Let $t$ be a positive integer not depending on $n$ and let $V$ be a subspace of $\mathbb{F}_2^n$ of dimension $t$ such that every $v\in V$ has $|v|\geq n^{8/15}$. Then the density of those $w\in B$ with $w\cdot v=0$ for all $v\in V$ is less than $(1.9)^{-(t-1)}$.
\end{corollary}

\begin{proof}
We shall be slightly sketchy about the some of the details when they are very standard. As always, we assume that $n$ is sufficiently large. Let $v_1,\dots,v_t$ be a basis given by Lemma \ref{Claim2} with $d=t$. Let $w$ be a random vector in $B$, let $i<t$, and let us consider the probability that $w.v_i=0$ given that $w.v_j=0$ for every $j<i$. 

The expected number of non-zero coordinates of $w$ in the union of the two intervals $I_i$ and $I_{i+1}$ is at least $n^{1/30}/2^{t-1}$, which tends to infinity, and the probability that it is at least half this number tends to 1 (very rapidly). If we condition further on this number, and if it is indeed at least $n^{1/30}/2^t$, then the probability that the number of non-zero coordinates of $w$ in $I_i$ is even is almost exactly 1/2. Therefore,  the probability that $w.v_i=0$ given that $w.v_j=0$ for every $j<i$ is less than $1/(1.9)$. 

Since this is true for every $i\leq t-1$, we obtain the result. 
\end{proof}

\begin{proof}[Proof of Lemma \ref{fourier}]
	Suppose that the result is not true. Let $r$ be a positive integer to be specified later and pick $R=\{u_1,\dots,u_r\}$ such that $\widehat{\mu_{B'}}(u_i)\geq 0.98$ for $i=1,2,\dots,r$. Then for each $i$, we have $u_i\cdot w=0$ for at least 99\% of all $w\in B'$. Therefore there is a subset $B''\subset B'$ with $|B''|\geq |B'|/2$ such that each $w\in B''$ has $u_i\cdot w=0$ for at least 98\% of the $u_i$. The number of subsets of $R$ of size $\frac{49}{50}r$ is at most ${r \choose 49r/50}={r \choose r/50}\leq (50e)^{r/50}\leq (1.8)^{49r/50}$. Let $t=49r/50$. Then there exists a subset $T$ of $R$ of size $t$ such that the number of $w\in B$ with $w\cdot u=0$ for all $u\in T$ is at least $\frac{|B''|}{(1.8)^t}\geq \frac{\delta}{2\cdot (1.8)^t}|B|$. Choose the smallest positive integer $t$ with $\frac{\delta}{2\cdot (1.8)^t}\geq (1.9)^{-(t-1)}$ (and with $r=50t/49$ an integer). Then the density of those $w\in B$ with $w\cdot u=0$ for all $u\in T$ is at least $(1.9)^{-(t-1)}$.
	
Now let $Q$ be the set of all $u\in \mathbb{F}_2^n$ with $\hat{\mu}_{B'}(u)\geq 0.98$ and assume that $|Q|\geq \exp(n^{2/3})$. Let $t$ and $r$ be as above and choose $u_1,\dots,u_r\in Q$ such that for every $j$, the (Hamming) distance of $u_j$ from $\text{span}(u_1,\dots,u_{j-1})$ is at least $n^{8/15}$. (This is possible because the number of $u\in \mathbb{F}_2^n$ with Hamming distance at most $n^{8/15}$ from an $r$-dimensional vector space is at most $2^r \exp(O(n^{8/15}\log n))<\exp(n^{2/3})$.) Applying Corollary \ref{Claim1} to $V=\text{span}(T)$, where $T$ is a subset of $\{u_1,\dots,u_r\}$ of size $t$, we find that the density of those $w\in B$ with $w\cdot u=0$ for all $u\in T$ is less than $(1.9)^{-(t-1)}$, which is a contradiction. 
\end{proof}

\begin{proof}[Proof of Lemma \ref{compatible}]
	In this proof, unless specified otherwise, we will view $\mathbb{F}_2^n$ as a subset of $\mathbb{R}^n$ and accordingly, the dot product is defined as $u\cdot w=\sum_i u(i)w(i)$ where the summation is in $\R$. Then $|u+w|\leq |u|$ is equivalent to $u\cdot w\geq |w|/2$. Hence $u$ is $B'$-compatible if $u\cdot w\geq |w|/2$ for at least $|B'|/3$ vectors $w\in B'$.
	
	Let $t$ be a fixed positive integer, not depending on $n$, to be specified later. For a multiset $T=\{u_1,\dots,u_t\}\subset A'$ write $s_T=\sum_{i=1}^tu_i-\frac{t}{2}q$ where $q$ is the vector in $\F_2^n$ consisting of ones. Let $a_k=s_T(k)$ and $\sigma_T^2=\sum_{k=1}^n a_k^2$. We say that $T$ is \emph{bad} if $\sigma_T^2\geq 1000tn$.
	
	\smallskip
	
	\noindent \emph{Claim 1.} If $T$ is not bad, then the number of $w\in B$ with $u_i\cdot w\geq |w|/2$ for all $i$ is at most $\frac{|B|}{100^t}$.
	
	\smallskip
	
	\noindent \emph{Proof of Claim 1.} If $u_i\cdot w\geq |w|/2$ for all $i$, then $s_T\cdot w\geq 0$. Note that $s_T\cdot w=\sum_{k\leq n} a_kw(k)$. We shall view $w$ as a random variable, chosen uniformly of all elements of $B$. What we need to prove is that $\P\lbrack \sum_{k\leq n} a_kw(k)\geq 0\rbrack \leq \frac{1}{100^t}$.
	
	Let $m=n^{1/2}$ and let $w_1,\dots,w_m$ be standard basis vectors of $\F_2^n$, chosen independently and uniformly at random. Note that the expected number of $i\neq j$ such that $w_i=w_j$ is at most 1, so almost surely this number is at most $\log n$. In particular, almost surely we have $n^{1/2}-2\log n\leq |w_1+\dots+w_m|\leq n^{1/2}$. Choose uniformly randomly an element $w\in B$ with minimal Hamming distance from $w_1+\dots+w_m\in \F_2^n$. This algorithm defines a uniformly random element of $w\in B$ such that almost surely we have $\sum_{k\leq n}|\sum_{i\leq m}w_i(k)-w(k)|\leq \sum_{k\leq n}|\sum_{i\leq m}w_i(k)-(\sum_{i\leq m}w_i)(k)|+\sum_{k\leq n}|(\sum_{i\leq m}w_i)(k)-w(k)| \leq 2\log n+2\log n= 4\log n$, where all the summations are taken in $\R$, except $\sum_{i\leq m} w_i$, which is taken in $\F_2^n$.

	At this point, we apply the following version of Chernoff's inequality, which appears as Theorem 3.4 in \cite{chunglu}.
	
	\smallskip
	
	Let $X_i$ ($1 \leq i \leq m$) be independent random variables satisfying $X_i \leq \E \lbrack X_i\rbrack + M$, for $1 \leq i \leq m$. We consider the sum $X=\sum_i X_i$
	with expectation $\E\lbrack X\rbrack = \sum_i \E\lbrack X_i\rbrack$ and variance $\text{Var}(X) = \sum_i \text{Var}(X_i)$. Then, we have $\P( X \geq \E\lbrack X\rbrack + \lambda) \leq \exp(-\frac{\lambda^2}{2(\text{Var}(X)+M\lambda/3)})$.
	
	\smallskip
	
	We now take $X_i=\sum_{k\leq n} a_kw_i(k)$ for $1\leq i\leq m$. Since $|a_k|\leq t$, the conditions of the theorem hold with $M=2t$. As $u_i\in A'$ for all $i$, we have $\sum_{k\leq n} a_k\leq t(n/2-10^{15}n^{3/4})-tn/2=-10^{15}tn^{3/4}$. Then $\E\lbrack X\rbrack=m\frac{\sum_{k\leq n} a_k}{n} \leq -\frac{1}{2}10^{15}tn^{1/4}$, and by the assumption that $T$ is not bad, $\text{Var}(X)\leq m \frac{\sum_{k\leq n} a_k^2}{n}\leq 1000tn^{1/2}$. Thus, taking $\lambda=10^{14}tn^{1/4}$ in the above theorem it follows that $$\P\lbrack \sum_{i\leq m}\sum_{k\leq n} a_kw_i(k)\geq -10^{14}tn^{1/4}\rbrack \leq \exp\bigg(-\frac{10^{28}t^2n^{1/2}}{2(1000tn^{1/2}+2t\cdot10^{14}tn^{1/4}/3)}\bigg)\leq \frac{1}{2\cdot 100^t}$$ But
	\begin{align*}
		\sum_{k\leq n}a_kw(k) &= \sum_{i\leq m,k\leq n}a_kw_i(k)+\sum_{k\leq n}a_k(w(k)-\sum_{i\leq m}w_i(k)) \\
		&\leq \sum_{i\leq m,k\leq n}a_kw_i(k)+t\sum_{k\leq n}|(w(k)-\sum_{i\leq m}w_i(k))|,
	\end{align*} and $\sum_{k\leq n}|(w(k)-\sum_{i\leq m}w_i(k))|\leq 4\log n$ almost surely, it follows that $\mathbb{P} \lbrack\sum_{k\leq n}a_kw(k)\geq 0 \rbrack\leq \frac{1}{100^t}$ and Claim 1 is proved.

	\smallskip
	
	\noindent \emph{Claim 2.} If $u_1,u_2,\dots,u_t$ are independently and uniformly randomly chosen elements of $A'$ then the probability that $T=\{u_1,\dots,u_t\}$ is bad is $o(\exp(-n^{7/8}))$.
	
	\smallskip
	
	\noindent \emph{Proof of Claim 2.} Recall that $T$ is bad if and only if $\sum_{k\leq n}(\sum_{i\leq t}u_i(k)-t/2)^2\geq 1000tn$. $u_1,\dots,u_t$ are randomly chosen from $A'$ but with probability $1-o(\exp(n^{-7/8}))$ all of them have $|u_i|\geq n/2-n^{99/100}$ so we may assume that $u_1,\dots,u_t$ are randomly chosen from the set $A''=\{v\in \F_2^n: n/2-n^{99/100}\leq |v|\leq n/2-10^{15}n^{3/4}\}$. It is not hard to see that we can write $u_i=x_i+y_i$ where $x_i$ and $y_i$ are random variables taking values in $\F_2^{n}$ and having the property that $x_i(k)$ are independent Bernoulli with parameter $1/2$ and $|y_i|\leq 2n^{99/100}$ with probability $1-o(\exp(-n^{7/8}))$. Then it suffices to prove that $$\P\lbrack\sum_{k\leq n}(\sum_{i\leq t}x_i(k)-t/2)^2\geq 500tn\rbrack=o(\exp(-n^{7/8}))$$
	Let $X_i=(\sum_{i\leq t} x_i(k)-t/2)^2$ for $1\leq k\leq n$. Then the $X_i$ are iid random variables with $\E \lbrack X_i\rbrack=t/4$ and $\text{Var}(X_i)=O(1)$. Thus, by Theorem 3.4 from \cite{chunglu} (which is the theorem stated above), taking $\lambda=100tn$ and $M=t^2$, it follows that $$\P\lbrack\sum_{k\leq n}(\sum_{i\leq t}x_i(k)-t/2)^2\geq 500tn\rbrack\leq \exp\bigg(-\frac{(100tn)^2}{2(nO(1)+t^2\cdot100tn/3)}\bigg)= o(\exp(-n^{7/8})),$$ finishing the proof of Claim 2.
	
	\smallskip
	
	We are now in a position to complete the proof of the lemma. Let $t$ be the smallest positive integer with $\frac{1}{100^t}<\frac{\d}{100(6e)^t}$. Let the density of $B'$-compatible elements in $A'$ be $\alpha$. Pick $v_1,\dots,v_{6t}$ independently and uniformly randomly from $A'$. Then with probability $\alpha^{6t}$, every $v_i$ is $B'$-compatible. If that is the case, then for every $i$, there are at least $|B'|/3$ vectors $w\in B'$ with $v_i\cdot w\geq |w|/2$. It follows that there is some $B''\subset B'$ with $|B''|\geq |B'|/100$ such that for every $w\in B''$ we have $v_i\cdot w\geq |w|/2$ for at least $t$ choices of $i$. The number of $t$-sets in $\{v_1,\dots,v_{6t}\}$ is at most $(6e)^{t}$ so there must exist a $t$-set $T=\{u_1,\dots,u_t\}\subset \{v_1,\dots,v_{6t}\}$ (multisets are allowed) such that the number of $w\in B$ with $u_i\cdot w\geq |w|/2$ for each $i$ is at least $|B''|/(6e)^{t}\geq\frac{\delta|B|}{100(6e)^{t}}> \frac{|B|}{100^t}$. By Claim 1, it follows that $T$ is bad. Thus, the probability that $T=\{u_1,\dots,u_t\}$ is bad when $u_1,\dots,u_t$ are independently and uniformly randomly chosen from $A'$, is at least $\frac{\a^{6t}}{{6t \choose t}}$. Hence, by Claim 2, we have $\frac{\a^{6t}}{{6t \choose t}}=o(\exp(-n^{7/8}))$, and we get $\alpha=o(\exp(-n^{3/4}))$.
\end{proof}


\thebibliography{99}

\bibitem{barak} B. Barak, P. K. Kothari and D. Steurer, \emph{Small-Set Expansion in Shortcode Graph and the 2-to-2 Conjecture}, arXiv preprint, \href{https://arxiv.org/abs/1804.08662}{arXiv:1804.08662}

\bibitem{chunglu} F. Chung and L. Lu, \emph{Concentration inequalities and martingale inequalities: a survey}, Internet Mathematics, 3(1), pp.79-127. (2006)

\bibitem{green} B. J. Green, \emph{The asymmetric Balog-Szemer\'edi-Gowers theorem}, \href{http://people.maths.ox.ac.uk/greenbj/papers/asym-BSG.pdf}{unpublished lecture notes}

\bibitem{janzer} O. Janzer, \emph{Low analytic rank implies low partition rank for tensors}, arXiv preprint

\bibitem{khot} S. Khot, \emph{On the power of unique 2-prover 1-round games}, STOC '02 Proceedings of the thirty-fourth annual ACM symposium on theory of computing (2002), pp. 767-775.

\bibitem{KMS} S. Khot, D. Minzer and M. Safra, \emph{Pseudorandom sets in Grassman graph have near-perfect expansion}, Electronic Colloquium on Computational Complexity, Report no. 6 (2018).

\bibitem{lovett} S. Lovett, \emph{The analytic rank of tensors and its applications}, \href{https://arxiv.org/abs/1806.09179}{arXiv:1806.09179}

\bibitem{taovu} T. C. Tao and V. Vu, \emph{Additive combinatorics}, Cambridge studies in advanced
mathematics \textbf{105}, CUP 2006.
\end{document}